\documentclass{article}
\usepackage[utf8]{inputenc}

\usepackage{amsmath}
\usepackage{amsthm}
\usepackage{amssymb}
\usepackage{anyfontsize}
\usepackage{bm}
\usepackage{enumerate}
\usepackage{esint}
\usepackage{etoolbox}
\usepackage{isomath}
\usepackage{mathrsfs}
\usepackage{mathtools}
\usepackage{multicol}
\usepackage{pgfplots}
\usepackage{scalerel}
\usepackage{stackengine}
\usepackage{tabulary}
\usepackage{tikz}

\newcommand{\ns}{\!\!}
\newcommand{\nss}{\ns\ns}

\newcommand{\qqquad}{\qquad\qquad}
\newcommand{\qqqquad}{\qqquad\qqquad}

\newcommand{\set}[1]{\mathcal{#1}}
\newcommand{\class}[1]{\mathscr{#1}}

\newcommand{\bs}{\setminus}

\newcommand{\haus}{\mathscr{H}}

\newcommand{\nats}{\mathbb{N}}
\newcommand{\whls}{\mathbb{W}}

\newcommand{\re}{\mathbb{R}}

\newcommand{\ren}{\re^n}

\newcommand{\reo}{\ov{\re}}

\newcommand{\bll}{\set{B}}

\newcommand{\cnt}{\class{C}}

\newcommand{\smfrac}[2]{{\textstyle{\frac{#1}{#2}}}}

\newcommand{\slint}[1]{\int_{#1}}
\newcommand{\sluint}[2]{\int_{#1}^{#2}}
\newcommand{\lint}[1]{\int\limits_{#1}}

\newcommand{\fslint}[1]{\fint_{#1}}
\newcommand{\flint}[1]{\fint\limits_{#1}}
\newcommand{\dd}{\mathrm{d}}

\makeatletter
\newcommand*\bdot{\mathpalette\bdot@{.65}}
\newcommand*\bdot@[2]{\mathbin{\vcenter{\hbox{\scalebox{#2}{$\m@th#1\bullet$}}}}}
\makeatother

\makeatletter
\newcommand*\bddot{\mathpalette\bddot@{.65}}
\newcommand*\bddot@[2]{\mathbin{\vcenter{\hbox{\scalebox{#2}
    {$\m@th#1\smash{{}_{\bullet}^{\bullet}}$}}}}}
\makeatother

\newcommand{\circled}[2][]{%
  \tikz[baseline=(char.base)]{%
    \node[shape = circle, draw, inner sep = .5pt]
    (char) {\phantom{\ifblank{#1}{#2}{#1}}};%
    \node at (char.center) {\makebox[0pt][c]{#2}};}}
\robustify{\circled}

\newcommand{\ve}[1]{\vectorsym{#1}}

\newcommand{\vx}{\ve{x}}
\newcommand{\vy}{\ve{y}}
\newcommand{\vz}{\ve{z}}

\newcommand{\lp}{\left(}
\newcommand{\rp}{\right)}

\newcommand{\lc}{\left\{}
\newcommand{\rc}{\right\}}

\newcommand{\lt}{\left.}
\newcommand{\rt}{\right.}

\def\loc{\operatorname{loc}}
\def\dist{\operatorname{dist}}

\def\diam{\operatorname{diam}}

\def\card{\operatorname{card}}

\newcommand{\dplus}{\mathbin{+\mkern-10mu+}}
\newcommand{\nd}{\quad\text{ and }\quad}
\newcommand{\frl}{\quad\text{ for all }}

\newcommand{\ov}[1]{\overline{#1}}
\newcommand{\ul}[1]{\underline{#1}}
\newcommand{\vep}{\varepsilon}

\newcommand{\wh}[1]{\widehat{#1}}

\newcommand{\supth}[1]{#1^\text{th}}

\stackMath
\newcommand\reallywidecheck[1]{%
\savestack{\tmpbox}{\stretchto{%
  \scaleto{%
    \scalerel*[\widthof{\ensuremath{#1}}]{\kern-.6pt\bigwedge\kern-.6pt}%
    {\rule[-\textheight/2]{1ex}{\textheight}}
  }{\textheight}%
}{0.5ex}}%
\stackon[1pt]{#1}{\scalebox{-1}{\tmpbox}}%
}

\usepackage{graphicx}
\usepackage{hyperref}
\usepackage[capitalise]{cleveref}

\newcommand{\dom}{\Omega}
\newcommand{\bnd}{\Gamma}

\newcommand{\vei}{\ve{i}}

\newcommand{\ovx}{\ov{\vx}}
\newcommand{\ovy}{\ov{\vy}}
\newcommand{\ovz}{\ov{\vz}}

\theoremstyle{plain}
\newtheorem{theorem}{Theorem}[section]
\newtheorem{lemma}[theorem]{Lemma}
\newtheorem{corollary}[theorem]{Corollary}
\theoremstyle{definition}
\newtheorem{definition}{Definition}[section]
\newtheorem{example}{Example}[section]
\theoremstyle{remark}
\newtheorem{remark}{Remark}

\title{Traces on General Sets in $\ren$ for Functions with no Differentiability Requirements\thanks{The author's research is supported by the NSF award DMS-1716790}}

\author{Mikil Foss\thanks{203 Avery Hall,
    Lincoln, NE 68588-0130 USA, University of Nebraska-Lincoln (mfoss3@unl.edu)}}

\begin{document}

\maketitle

\begin{abstract}
   This paper is concerned with developing a theory of traces for functions that are integrable but need not possess any differentiability within their domain. Moreover, the domain can have an irregular boundary with cusp-like features and codimension not necessarily equal to one, or even an integer. Given $\dom\subseteq\ren$ and $\bnd\subseteq\partial\dom$, we introduce a function space $\class{N}^{s(\cdot),p}(\dom)\subseteq L^p_{\loc}(\dom)$ for which a well-defined trace operator can be identified. Membership in $\class{N}^{s(\cdot),p}(\dom)$ constrains the oscillations in the function values as $\bnd$ is approached, but does not imply any regularity away from $\bnd$. Under connectivity assumptions between $\dom$ and $\bnd$, we produce a linear trace operator from $\class{N}^{s(\cdot),p}(\dom)$ to the space of measurable functions on $\bnd$. The connectivity assumptions are satisfied, for example, by all $1$-sided nontangentially accessible domains. If $\bnd$ is upper Ahlfors-regular, then the trace is a continuous operator into a Sobolev-Slobodeckij space. If $\bnd=\partial\dom$ and is further assumed to be lower Ahlfors-regular, then the trace exhibits the standard Lebesgue point property. To demonstrate the generality of the results, we construct $\dom\subseteq\re^2$ with a $t>1$-dimensional Ahlfors-regular $\bnd\subseteq\partial\dom$ satisfying the main domain hypotheses, yet $\bnd$ is nowhere rectifiable and for every neighborhood of every $\ovx\in\bnd$, there exists a boundary point within that neighborhood that is only tangentially accessible.\\[5pt]
 {\it AMS2010}:
 35A23, 46E35, 47G10

 \end{abstract}


\section{Introduction}
\subsection{Overview}\label{SS:Overview}
Suppose that $\dom\subseteq\ren$ is an open, not necessarily bounded, set and that $u\in\cnt(\dom)$ is uniformly continuous. Though the boundary $\partial\dom$ is not in the domain of $u$, there is a natural choice for a trace function $Tu\in\cnt(\partial\dom)$ that can be identified as the values of $u$ on $\partial\dom$, since $u$ has a continuous extension to $\ov{\dom}$. In the context of Sobolev spaces, where $u$ is not necessarily continuous in $\dom$, the well-known Gagliardo's theorem~\cite{Gag:57a} states that, with a sufficiently regular boundary, this trace operator can be extended from the space of functions with uniformly continuous derivatives to a continuous linear operator $T:W^{1,p}(\dom)\to W^{1-\frac{1}{p},p}(\partial\dom)$, for each $1<p<\infty$. Here $0<\beta<1$ and $W^{\beta,p}(\partial\dom)$ is the standard Sobolev–Slobodeckij space (\cref{D:SobSlo}). Generalizations to Sobolev-Slobodeckij and Besov spaces, traces on general closed subsets of $\ren$, Sobolev and Besov spaces on metric spaces, Sobolev spaces with variable exponent, etc. have since been produced (see~\cite{DieHar:11a,JonWal:84a,Mar:18a,Mar:87a,SakSot:17a}). This paper is concerned with establishing analogous trace results compatible with a nonlocal framework, where only $p$-integrability is assumed away from the boundary.

There has been a surge of interest in developing and analyzing models that employ a nonlocal operator. Such models can incorporate long-range interactions and multiple scales and can expand the space of admissible solutions to permit singular and discontinuous attributes. Nonlocal models have been successfully employed in a wide variety of contexts, including image processing~\cite{GilOsh:09a,KatFoi:10a}, population and flocking models~\cite{CovDup:07a,ShyTad:19a}, diffusion~\cite{AndMazRos:10a}, phase transitions~\cite{BatHan:05a,GalGioGras:17a}, and material deformation with failure in peridynamics models~\cite{Sil:00a,Sil:18a}. For a comprehensive introduction to nonlocal modeling and their analysis, we refer to the monograph~\cite{Du:19a}.

A common class of nonlocal operators have a convolution, or cross-correlation, like structure. At each point $\vx\in\dom$, the operator uses an integral and an integrable kernel to accumulate weighted data for a function over a neighborhood of $\vx$. This makes them generally insensitive to function values on sets of zero measure and, in particular, to function irregularities across sets of dimension less than $n$. With an appropriate kernel, however, the operator can approximate a differential operator and provide information about the rate of change of a function.

The current work contributes to the rigorous development of a framework in which convolution-like operators, with integrable kernels, can incorporate data on sets with positive codimension. In addition to mathematical interest, there are two primary motivations for these efforts. Both are related to the fact that at points where the support of the kernel extends outside of $\dom$, the evaluation of the operator requires function values outside of $\dom$. Thus, for associated nonlocal problems, the analogue of a Dirichlet-type boundary condition is a volume-constraint, where the value of a solution is prescribed on a region of positive measure in $\ren\bs\dom$~\cite{DuGunLeh:11a,Ros:16a}. It can be a nontrivial issue to identify appropriate volume-constraints when data on only a lower-dimensional set $\bnd\subseteq\partial\dom$ is readily available. If the integral operator is responsive to function behavior on $\bnd$, then one can instead formulate nonlocal problems subject to classical Dirichlet-type boundary conditions. This also facilitates seamless transitions between nonlocal and local system descriptions. There is substantially more computational expense to numerically solve nonlocal equations when compared to their local counterpart. With a nonlocal operator that ``localizes'' as $\bnd$ is approached, one can couple a computationally efficient local model to a nonlocal model that confined to a region where it is essential to employ nonlocal operators~\cite{DelLiSel:19a,TaoTiaDu:19a}. The set $\bnd$ is a lower-dimensional interface between these two regions. In both contexts, it is essential to have a trace theory to ensure well-posedness and mathematical consistency.





In~\cite{TiaDu:17a}, Du and Tian produced some of the first trace results for functions in the nonlocal setting, with operators that have an integrable kernel that concentrates near the boundary. They considered the space $\class{S}(\dom)\subseteq L^2(\dom)$ consisting of functions satisfying $\|u\|^2_{\class{S}(\dom)}:=\|u\|^2_{L^2(\dom)}+|u|^2_{\class{S}(\dom)}<\infty$,
where $\dom\subseteq\ren$ is an open Lipschitz domain and
\[
    |u|^2_{\class{S}(\dom)}
    :=
    \int_\dom\fslint{\Psi(\vx)}
    \lp\frac{|u(\vy)-u(\vx)|}{d_{\partial\dom(\vx)}}\rp^2\dd\vy\dd\vx
    \footnote{For the purposes of comparison, this is a simplified norm equivalent to the one defined in~\cite{TiaDu:17a}.}.
\]
Here
\[
    \fslint{E}v(\vx)\dd\vx:=\frac{1}{|E|}\slint{E}v(\vx)\dd\vx,
    \frl E\in\class{B}(\ren),\text{ with }|E|>0\text{ and }v\in L^1(E),
\]
\[
    d_{\partial\dom}(\vx):=\inf_{\ovx\in\partial\dom}\|\ovx-\vx\|_{\ren},
    \nd
    \Psi(\vx):=\bll_{\frac{1}{2}d_{\partial\dom}(\vx)}(\vx),
    \frl\vx\in\ren,
\]
where $\class{B}(\ren)$ denotes the family of Borel subsets of $\ren$, $|E|$ is the Lebesgue measure of $E$, and $\bll_\rho(\vx)$ is the open ball centered at $\vx$ with radius $\rho>0$. The kernel, in the semi-norm $|\cdot|_{\class{S}(\dom)}$, can be identified as $\vx\mapsto d_{\partial\dom}(\vx)^{-n-2}|\bll_\frac{1}{2}|^{-1}\chi_{\Psi(\vx)}$. Since $d_{\partial\dom}$ is uniformly positive on any open $\dom'$ compactly contained in $\dom$, one can only expect $u\in L^2(\dom')$. Thus, the space $\class{S}(\dom)$ includes functions that have no regularity, beyond square-integrability, away from the $\partial\dom$. As the $\partial\dom$ is approached, however, the kernel concentrates sufficiently strongly that the behavior of $u$ at $\partial\dom$ contributes to $|\cdot|_{\class{S}(\dom)}$. In fact, the singularity is strong enough that $|\cdot|_{\class{S}(\dom)}$ is sensitive to oscillations in $u$ at the boundary and $\|u\|_{\class{S}(\dom)}<\infty$ implies that there is a well-defined trace $Tu\in L^2(\partial\dom)$ and, moreover, that $Tu\in W^{\frac{1}{2},2}(\partial\dom)$. These results have recently been generalized to exponents $1\le p<\infty$ in~\cite{DuMenTia:20a}.

In this paper we vastly expand the class of functions and domains, for which a well-defined trace can be identified for a given $\bnd\in\class{B}(\partial\dom)$. Let $1\le p<\infty$ and $s\in L^\infty(\dom)$ be given. For each $u\in L^1(\dom)$, define
\begin{equation}\label{D:nuDef}
  \nu^{s(\cdot),p}(E;u)
  :=
  \slint{\dom\cap E}\lp\fslint{\Psi(\vx)}
    \frac{|u(\vy)-u(\vx)|}{d_{\partial\dom}(\vx)^{s(\vx)}}\dd\vy\rp^p\dd\vx,
  \frl E\in\class{B}(\ren).
\end{equation}
The focus of this paper is on developing a trace theory for functions belonging to the space
\[
  \class{N}^{s(\cdot),p}(\dom)
  :=
  \lc u\in L^1(\dom):\nu^{s(\cdot),p}(\dom;u)<\infty
  \rc.
\]
A straightforward argument shows that $|\cdot|_{\class{N}^{s(\cdot),p}(\dom)}:\class{N}^{s(\cdot),p}(\dom)\to[0,\infty)$, given by $|u|^p_{\class{N}^{s(\cdot),p}(\dom)}:=\nu^{s(\cdot),p}(\dom;u)$, provides a semi-norm. We easily see that if for some $0<\beta'<1$, we have $s\ge\beta'$ throughout $\dom$, then $W^{\beta',p}(\dom)\subseteq B^{1,p}_{\beta'}(\dom)\subseteq\class{N}^{s(\cdot),p}(\dom)$, with $B^{1,p}_{\beta'}(\dom)$ is a Besov space, as defined in~\cite{JonWal:84a}. Moreover, the spaces introduced in~\cite{DuMenTia:20a,TiaDu:17a} are subspaces of the corresponding space $\class{N}^{s(\cdot),p}$, with $s$ constant. (For example, $\class{S}\subseteq\class{N}^{1,2}$.) In fact, \cref{E:StrictContain} shows that, in general, the containment is strict. Thus, many results currently available in the literature can be viewed as corollaries of the trace theorems established in this paper.

With $\bnd\in\class{B}(\partial\dom)$, the main results of the paper provide assumptions on $\dom$, $\bnd$, and $s$ (collected in Section~\ref{SS:Assumptions}) that ensure {\it the existence of a continuous linear trace operator} $T:\class{N}^{s(\cdot),p}(\dom)\to W^{\beta,p}(\bnd)$, for some $0<\beta<1$. The primary assumptions on $\dom$ and $\bnd$ are (H1) a corkscrew-type condition, (H2) a uniform-connectedness condition, (H3) a constraint on the oscillations of $u$ near $\bnd$, and (H4) Ahlfors-regularity of $\bnd$. Loosely speaking, the oscillation constraint is satisfied if $s(\vx)\ge a(\ovx)+\beta+\frac{n-\dim(\bnd)}{p}$, for $\vx\in\dom$ near $\ovx\in\bnd$. The values of $a(\ovx)>0$ depend, in an explicit way, on the approachability of $\ovx$ from $\dom$. Following some definitions recorded in Section~\ref{SS:Definitions}, precise statements of the main assumptions and theorems are given in Sections~\ref{SS:Assumptions} and~\ref{SS:MainThms}.

The distinguishing features of this work are the following:
\begin{itemize}
\item {\bf We provide trace results allowing ``very rough'' boundaries.} In~\cite{DuMenTia:20a} and~\cite{TiaDu:17a}, the set $\bnd$ is required to be Lipschitz, so that an atlas of Lipschitz transforms are available to ``flatten out'' $\bnd$. They connect the rate of change of a function in directions parallel to the flattened boundary to the rate of change in the normal direction. In this paper, a different approach is presented, based on a continuous extension of a function~\eqref{D:gMeanFun} related to the mean values of $u$. The corkscrew and connectedness properties ensure there is a region of approach to points in $\bnd$ that allows us to work directly with $\bnd$, without any transformations, and allow $\bnd$ to have minimal regularity.

\hspace{17pt}In fact, identifying a natural candidate for a measurable trace does not require {\it any} assumptions on $\dom$ and $\bnd$ beyond (H1) and (H2). Assuming upper Ahlfors-regularity, we prove $Tu\in W^{\beta,p}(\bnd)$. If, additionally, there is a lower Ahlfors-regular neighborhood in $\bnd$, then we can establish the Lesbesgue point property for $Tu$ within that neighborhood; i.e., as $\rho\to0^+$, the mean values of $u$ over $\dom\cap\bll_\rho(\ovx)$ converge in the $p$-norm to $Tu(\ovx)$, for a.e. $\ovx\in\bnd$. Thus, with respect to the surface measure, $Tu$ agrees a.e. on $\bnd$ with the strictly defined function associated with $u$.

\item {\bf Within $\dom$, the functions need only be $p$-integrable.} Given a function $u\in\class{N}^{s(\cdot),p}(\dom)$ and an open set $\dom'$ compactly contained in $\dom$, Jensen's inequality implies
\[
    \slint{\dom'}\!|u(\vx)|^p\dd\vx
    \le
    c\slint{\dom'}\!\lp\fslint{\Psi(\vx)}\ns|u(\vy)-u(\vx)|\dd\vy\rp^p\dd\vx\\
    +c\slint{\dom'}\!\lp\fslint{\Psi(\vx)}\ns|u(\vy)|\dd\vy\rp^p\dd\vx
    <\infty,
\]
since $u\in L^1(\dom)$ and $d_{\partial\dom}(\vx)$ is uniformly positive in $\dom'$. Hence, $\class{N}^{s(\cdot),p}(\dom)\subseteq L^p_{\loc}(\dom)$. As with the space $\class{S}(\dom)$, however, no additional regularity can be expected for $u$ in $\dom'$, regardless of the behavior of $s$ in $\dom'$. This makes $\class{N}^{s(\cdot),p}$ a viable solution space for nonlocal systems where even discontinuous functions are admissible. For models of phenomena exhibiting sharp transitions or jumps, it is critical to include irregular functions as solution candidates. Moreover, functions in this space possess well-defined fine properties as we approach $\bnd$.

\item {\bf The traces are captured on possibly ``very thin or porous'' sets.} The set $\bnd$ can have non-integer Hausdorff dimension, any positive codimension, and may also possess cusp-like features. To demonstrate how irregular the domain can be, we produce an $\dom\subseteq\re^2$, with a nonrectifiable Ahlfors-regular self-similar set $\bnd\in\class{B}(\partial\dom)$ that has Hausdorff dimension $1<t<2$ and has the following property: for every $\ovx\in\bnd$ and $\rho>0$, there exists a $\ovy\in\bnd\cap\bll_\rho(\ovx)$ that is only tangentially accessible. In other words, the cone of directions in the approach region for $\ovy$ degenerates as it is approached. Nevertheless, provided the oscillation constraint is satisfied, there is a trace $Tu\in W^{\beta,p}(\bnd)$ possessing the Lebesgue point property (see \cref{E:Domains}(\ref{E:Prickly}), \cref{R:ThmRmks}(\ref{R:PricklyExample}), and Section~\ref{S:Appendix}). We mention the recent interest in developing a theory for elliptic problems with higher-codimensional boundaries~\cite{DavFenMay:19a,MayGuyFen:18a,FenMayZha:18a,MayZih:19a} and trace theorems and boundary value problems on fractal sets~\cite{AchSabTch:06a,AchTch:10a,Cap:13a}.


\end{itemize}

\subsection{Definitions}\label{SS:Definitions}

A more detailed presentation of the main theorems requires some additional definitions. For each $\rho>0$, $E\subseteq\ren$, and $\vx\in\ren$, define $E_\rho(\vx):=E\cap\bll_\rho(\vx)$. Given $\tau\ge0$ and $E\in\class{B}(\ren)$, we use $\haus^\tau(E)$ to denote the $\tau$-dimensional Hausdorff measure of $E$. We use $L(E)$ for the space of Borel-measurable functions. In general, by measurable, we mean Borel-measurable. We now introduce the two main geometric properties needed throughout the paper. A discussion of these definitions with accompanying figures is provided in the next section, following the main assumptions.
\begin{definition}
With $0<\eta<1\le\theta<\infty$, we will say that $Q\subseteq\dom$ is an \emph{$(\eta,\theta)$-corkscrew region} for $\ovx\in\partial\dom$ if there exists a $\delta_0>0$ such that, for each $0<\delta\le\delta_0$, there exists $\vx\in Q$ satisfying
\[
    \eta\delta<\|\ovx-\vx\|_{\ren}<\delta
    \nd
    d_{\partial\dom}(\vx)>(\eta\delta)^\theta.
\]
We refer to $\delta_0$ as the \emph{radius of $Q$}.
\end{definition}
\begin{definition}
With $C,\theta\ge 1$, we will say that $E\subseteq\dom$ is \emph{$(C,\theta)$-connected} to $\ovx\in\partial\dom$ if there exists $\rho_0>0$ with the following property: for each $0<\lambda<1$, there exists $0<\vep_\lambda\le\lambda$ such that for any $0<\rho\le\rho_0$, if
\[
    \vx,\vx'\in E_\rho(\ovx)
    \nd
    d_{\partial\dom}(\vx),d_{\partial\dom}(\vx')\ge(\lambda\rho)^\theta,
\]
then there exists a rectifiable path $\ve{\gamma}:[0,1]\to E$ between $\vx$ and $\vx'$ such that
\begin{equation}\label{H:PathProp}
    \haus^1(\ve{\gamma}([0,1]))\le C\rho
    \nd
    d_{\partial\dom}(\ve{\gamma}(\tau))\ge\vep_\lambda\rho^\theta,
    \frl\tau\in[0,1].
\end{equation}
\end{definition}
In the definition above, without loss of generality, we assume $\vep_{\lambda}\le\vep_{\lambda'}$, if $\lambda\le\lambda'$. Next we recall the definition of the Sobolev-Slobodeckij spaces.
\begin{definition}\label{D:SobSlo}
Let $E\in\class{B}(\ren)$ be a $\haus^\tau$-dimensional set. For each $1\le p<\infty$ and $\beta>0$, define $|\cdot|_{W^{\beta,p}(E)}:L(E)\to[0,\infty]$ by
\[
    |u|^p_{W^{\beta,p}(E)}
    :=
    \slint{E}\slint{E}\frac{|u(\vy)-u(\vx)|^p}{\|\vy-\vx\|^{\tau+\beta p}_{\ren}}
        \dd\haus^\tau(\vy)\dd\haus^\tau(\vx)
\]
and the \emph{Sobolev-Slobodeckij space}
\[
    W^{\beta,p}(E):=\lc u\in L^p(E):|u|_{W^{\beta,p}(E)}<\infty\rc.
\]
\end{definition}
If $E\subseteq\ren$ is a closed set and $0<\beta<1$, then our definition of $W^{\beta,p}(E)$ corresponds to the definition of the Besov space $B^{p,p}_{\beta}(E)$ given in~\cite{Jon:94a} (see also~\cite{JonWal:84a}). We note that, in general, \cref{D:SobSlo} does not provide the standard Sobolev or Sobolev-Slobodeckij spaces when $\beta\ge 1$. In fact, if $E$ is an open connected set, $\beta\ge1$, and $u\in W^{\beta,p}(E)$, as defined above, then $u$ is a constant function~\cite{Bre:02a,MarMar:08a}. It appears to be unknown whether this is also true for $E$ not open.

Finally, we need a bit more notation. For each $0<\lambda<1\le\theta$ and $\ovx\in\partial\dom$, set
\[
    Q^\theta_\lambda(\ovx)
    :=
    \lc
        \vx\in\dom
        : d_{\partial\dom}(\vx)>\lp\lambda\|\ovx-\vx\|_{\ren}\rp^\theta
    \rc.
\]
(See \cref{F:Approach} below.) Given $\delta>0$, we also define $Q^\theta_{\lambda,\delta}(\ovx):=Q^\theta_\lambda\cap\bll_\delta(\ovx)$. For convenience, we define $\alpha\in\cnt(\re^2)$ by
\[
    \alpha(s,\theta):=\lc\begin{array}{ll}
        p(1-\theta)+(ps-n)\theta, & ps-n\le p-1,\\
        (1-\theta)+(ps-n), & ps-n> p-1.
    \end{array}\rt
\]
We note that $\alpha$ is piecewise bilinear and that $\alpha(\cdot,\theta)$ is increasing, for each $\theta\ge 1$.  For the remainder of the paper, we fix the parameters $0\le t<n$ and $0<\delta_\bnd<1$ and the uniformly bounded measurable functions $s:\dom\to[0,\infty)$ and $\theta_\bnd:\bnd\to[1,\infty)$. Put $\ov{\theta}_\bnd:=\sup_{\ovx\in\bnd}\theta_\bnd(\ovx)$. For each $0<\delta\le\delta_\bnd$, define $\ul{s}_\delta,\ul{\alpha}_\delta,\ul{\alpha}_0:\bnd\to\re$ by
\[
    \ul{s}_\delta(\ovx):=\inf_{\vx\in\dom_\delta(\ovx)}s(\vx),\quad
    \ul{\alpha}_\delta(\ovx):=\alpha(\ul{s}_\delta(\ovx),\theta_\bnd(\ovx)),
    \nd
    \ul{\alpha}_0(\ovx):=\lim_{\delta\to0^+}\ul{\alpha}_\delta(\ovx).
\]
We observe that the above functions are each measurable and uniformly bounded.

\subsection{Assumptions}\label{SS:Assumptions}
We now list our primary assumptions for $\dom$, $\bnd$ and the space $\class{N}^{s(\cdot),p}(\dom)$. We assume that $\bnd\in\class{B}(\partial\dom)$ has $\haus^t$-dimension. When convenient, we will just write $Q_{\lambda}(\ovx)$ and $Q_{\lambda,\delta}(\ovx)$ for $Q^{\theta_\bnd(\ovx)}_\lambda(\ovx)$ and $Q^{\theta_\bnd(\ovx)}_{\lambda,\delta}(\ovx)$, respectively.
\begin{itemize}
\item[(H1)] \textbf{Uniform $\theta_\bnd$-Corkscrew Condition}: There exists $0<\lambda_0,\eta_0<1$ such that, for each $\ovx\in\bnd$, the set $Q_{\lambda_0}(\ovx)$ is an $(\eta_0,\theta_\bnd(\ovx))$-corkscrew region for $\ovx$ with radius $\delta_\bnd$.
\item[(H2)] \textbf{Uniform $\theta_\bnd$-Connectedness Condition}: There exists $C_\bnd\ge 1$ such that, for each $\ovx\in\bnd$, the set $\dom_{\delta_\bnd}(\ovx)$ is $(C_\bnd,\theta_\bnd(\ovx))$-connected to $\ovx$.
\item[(H3)] \textbf{Oscillation Constraints}:
\begin{itemize}
    \item[(H3$'$)] \textbf{Pointwise Oscillation Constraint at $\bnd$}:
\[
    \ul{\alpha}_0(\ovx)>-t,
    \quad\text{ for $\haus^t$-a.e. }\ovx\in\bnd.
\]
    \item[(H3$''$)] \textbf{Uniform Oscillation Constraint near $\bnd$}:
\[
    \ul{\alpha}_{\bnd}:=\inf_{\ovx\in\bnd}\ul{\alpha}_{\delta_\bnd}(\ovx)>-t.
\]
\end{itemize}
\item[(H4)] \textbf{Ahlfors-Regularity}: There exists $A_\bnd\ge 1$ such that, for each $\ovx\in\bnd$,
\begin{itemize}
    \item[(H4$'$)] \textbf{Upper Ahlfors-Regularity}:
\[
    \haus^t(\bnd\cap\bll_\rho(\ovx))\le A_\bnd\rho^t,
    \frl \rho>0.
\]
    \item[(H4$''$)] \textbf{Lower Ahlfors-Regularity}:
\[
    A^{-1}_\bnd\rho^t\le\haus^t(\bnd\cap\bll_\rho(\ovx)),
    \frl 0<\rho\le\diam(\bnd).
\]
\end{itemize}
\end{itemize}

\begin{figure}[h]
    \centering
    \parbox[b]{2.6in}{
    \scalebox{1.1}{\includegraphics[trim=150pt 520pt 310pt 125pt,clip]{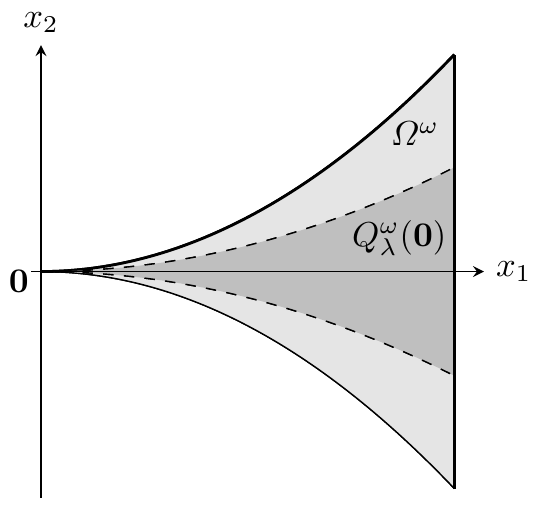}}
    \caption{Approach region $Q^\omega_\lambda(\ve{0})$ for $\ve{0}$}
    \label{F:Approach}}
    \hspace{20pt}
    \parbox[b]{2.6in}{
    \scalebox{.43}{\includegraphics[trim=10pt 17pt 10pt 15pt,clip]{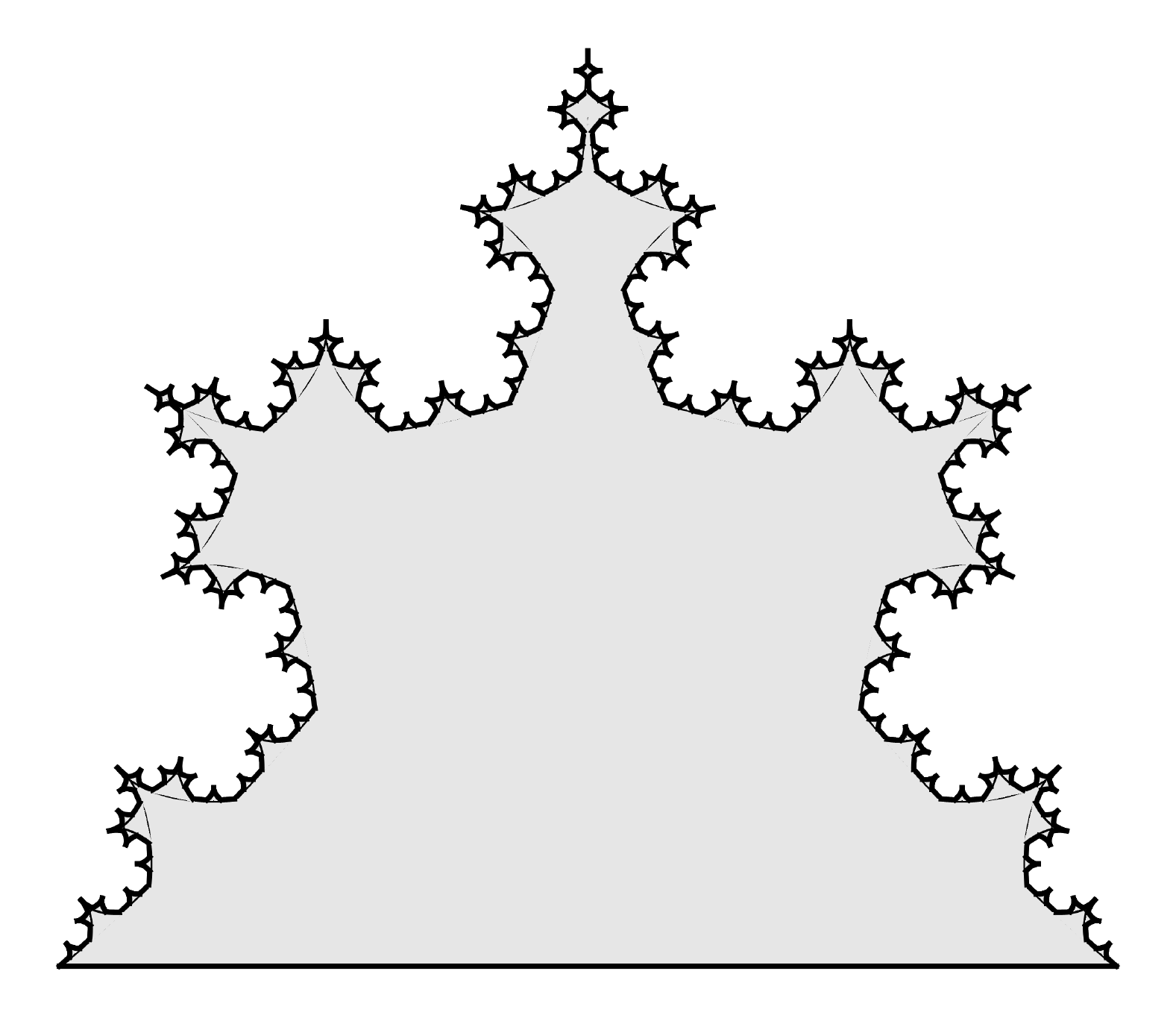}}
    \caption{Snowflake with cusps}
    \label{F:Snowflake}
    }
\end{figure}

\begin{example}\label{E:Domains}
\begin{enumerate}[(a)]
\item\label{E:Proto} The prototypical example for $Q^{\theta_\bnd(\ovx)}_\lambda(\ovx)$ is a wedge-like region where $\ovx$ is at a cusp. Let $\theta_0\ge1$ and $\frac{1}{2}<H\le\frac{\sqrt{3}}{2}$ be given. Define
\[
    \dom^{\theta_0}=\lc\vx=(x_1,x_2)\in\re^2:0<x_1<H
    \text{ and }0<|x_2|<\smfrac{1}{2}\lp x_1/H\rp^{\theta_0}\rc.
\]
Then, for any $0<\lambda_0,\eta_0<1$, the hypotheses (H1) and (H2) are satisfied with $\bnd=\partial\dom^{\theta_0}$ and $\theta_\bnd:=\theta_0\chi_{\{\ve{0}\}}+\chi_{\bnd\bs\{\ve{0}\}}$. Figure 1 depicts the approach region $Q^{\theta_\bnd(\ve{0})}_{\lambda_0}(\ve{0})=Q^{\theta_0}_{\lambda_0}(\ve{0})$ when $\theta_0>1$.
\item\label{E:Prickly} As mentioned in the introduction, assumptions (H1) and (H2), themselves, do not imply any regularity of the boundary. The Koch snowflake is a well-known example of an nontangentially accessible (NTA) domain (see \cref{R:HypConn}(\ref{R:NTA}) and~\cite{Fau:95a}) that has a nonrectifiable boundary but satisfies (H1) and (H2), with $\theta_\bnd\equiv1$ and $\vep_\lambda$ independent of $\lambda$. The Koch snowflake can be generated as the union of an iteratively produced sequence of polygonal domains. The initial domain is an equilateral triangle, for example $\dom^1$ with $H=\sqrt{3}/2$ defined in \cref{E:Domains}(\ref{E:Proto}) above. Subsequent domains are obtained by replacing the middle third of the boundary line segments with an appropriately scaled outward-pointing equilateral triangle. With a similar procedure, using $\dom^{\theta_0}$ with $\theta_0>1$, we can produce a ``prickly'' version of the snowflake domain (see \cref{F:Snowflake}). Taking $\bnd$ to be the fractal portion of the resulting domain boundary, we find that its Hausdorff dimension is $(\ln4/\ln3)\le t<2$ and $0<\haus^t(\bnd)<\infty$. Moreover, $\dom$ and $\bnd$ satisfy hypotheses (H1) and (H2), with $\theta_\bnd\equiv\theta_0$, and both (H4${}^\prime$) and (H4${}^{\prime\prime}$). The set $\dom$, however, fails to be a $1$-sided NTA domain. More specifically, for every $\rho>0$ and $\ovx\in\bnd$, there exists $\ovy\in\bnd\cap\bll_\rho(\ovx)$ such that there is no $(C,\theta)$-connected $(\eta,\theta)$-corkscrew region with positive radius for $\ovy$, for any $0<\eta<1\le\theta<\theta_0$ and $C\ge 1$. Thus, in particular, $\dom_\rho(\ovx)$ is not a $1$-sided NTA domain, for any $\rho>0$. Additional details are provided in the appendix (Section~\ref{S:Appendix}).
\end{enumerate}
\end{example}

We next put the corkscrew and connectedness assumptions into a broader context.
\begin{remark}\label{R:HypConn}
Together, (H1) and (H2) ensure each point $\ovx\in\bnd$ can be approached along a path $\ve{\gamma}\subseteq\dom$ with a quantitative control on the distance between points $\vx\in\ve{\gamma}$ and $\partial\dom$. For the rest of this remark, suppose that $\bnd=\partial\dom$ is bounded and that $\theta_\bnd\equiv1$ on $\partial\dom$.
\begin{enumerate}[(a)]
\item\label{R:NTA} In this setting, hypotheses (H1) is the standard (interior) corkscrew condition. If $\dom$ also possesses the (interior) Harnack chain property, then $\dom$ is said to be a \emph{$1$-sided NTA domain}. An NTA domain is a $1$-sided NTA domain such that $\ren\bs\dom$ also satisfies the corkscrew condition. These domains were introduced in~\cite{JerKen:82a} in connection to the absolute continuity of the harmonic measure with respect to the surface measure on $\partial\dom$. As indicated in \cref{E:Domains}(\ref{E:Prickly}), NTA domains need not even have rectifiable boundaries. Recently, it has been shown that if $\dom$ is $1$-sided NTA and $\partial\dom$ is both $(n-1)$-dimensional and upper and lower Ahlfors-regular, then rectifiability of $\partial\dom$ is actually equivalent to the absolute continuity of the harmonic measure~\cite{AzzMar:17a}.
\item\label{R:H2Rem}  The motivation for (H2) is Lemma 2.1 in~\cite{MayGuyFen:18a}. As part of their investigation of elliptic problems in domains with higher codimensional boundaries, they show that Ahlfors-regular sets in $\ren$, with dimension $0\le t<n-1$, satisfy (H2), again with $\theta_\bnd\equiv 1$. There is a close relation between assumptions (H1) and (H2) and local uniformity. The domain $\dom$ is \emph{locally uniform} if there exists $C\ge 1$ and $\rho_0>0$ such that, for every $\vx,\vx'\in\dom$ satisfying $\|\vx-\vx'\|_{\ren}\le\rho_0$, there exists a rectifiable path $\ve{\gamma}:[0,1]\to\dom$ between $\vx$ and $\vx'$ such that
\[
    \haus^1(\ve{\gamma}([0,1]))\le C\|\vx-\vx'\|_{\ren}
\]
and
\[
    d_{\partial\dom}(\ve{\gamma}(\tau))
    \ge
    C^{-1}\min\{\haus^1(\ve{\gamma}([0,\tau])),\haus^1(\ve{\gamma}([\tau,1]))\}.
\]
We see that local uniformity implies (H2) with $\vep_\lambda$ independent of $\lambda$. Assumptions (H1) and (H2) together, however, imply $\dom$ is locally uniform (see Lemma~A.1 and the proof for Theorem 2.15 in~\cite{AzzMar:17a}). A condition equivalent to local uniformity is the $(\vep,\delta)$-condition. It has been shown that, on $(\vep,\delta)$-domains, there exists a linear continuous extension operator for the BMO and Sobolev mappings~\cite{Jon:80a,Jon:81a}. If $\dom$ is locally uniform with $\rho_0\ge\diam(\dom)$, then it is a \emph{uniform} domain. These domains were introduced in~\cite{MarSar:79a}, for which injectivity and approximation results for locally bi-Lipschitz mappings were established. The class of uniform domains is actually equivalent to the class $1$-sided NTA. In view of part (a) above, we see that assumptions (H1) and (H2) are both satisfied, with $\theta_\bnd\equiv1$, by $1$-sided NTA domains. We point out that, in \cref{E:Domains}(\ref{E:Prickly}), there are no locally uniform neighborhoods of any point in $\bnd$.
\end{enumerate}
\end{remark}

Assumptions (H3) can, in some sense, be interpreted as a requirement that the oscillations in $u\in\class{N}^{s(\cdot),p}(\dom)$ decay as $\bnd$ is approached. If $\nu^{s(\cdot,p}(\dom;u)<\infty$, then there exists $f\in L^p(\dom)$ such that
\[
    \fslint{\Psi(\vx)}|u(\vy)-u(\vx)|\dd\vy\le d_{\partial\dom}(\vx)^{ps(\vx)}f(\vx),
    \quad\text{ for a.e.-}\vx\in\dom.
\]
The left-hand side ``measures'' the deviation of the values of $u$ from $u(\vx)$ over the ball $\Psi(\vx)$. Since $t<n$, (H3) implies that, for each $\ovx\in\bnd$, there exists a $\delta_0>0$ such that $s$ is uniformly positive within the corkscrew region $Q_{\lambda_0,\delta_0}(\ovx)$. Thus, (H3) ensures the oscillations of $u$ dampen in $Q_{\lambda_0,\delta}(\ovx)$, as $\delta\to0^+$, which allows a well-defined value for $u(\ovx)$ to be identified.

While Jensen's inequality implies, for all $1\le q<\infty$, we have
\[
    \nu^{s(\cdot),p}(\dom;u)
    \le
    \nu^{s(\cdot),(p,q)}(\dom;u)
    :=
    \slint{\dom\cap U}\lp\fslint{\Psi(\vx)}
    \frac{|u(\vy)-u(\vx)|^q}{d_{\partial\dom}(\vx)^{qs(\vx)}}\dd\vy\rp^\frac{p}{q}\dd\vx.
\]
As just discussed, assumption (H3) is a type of decay requirement for the oscillations of $u\in\class{N}^{s(\cdot),p}(\dom)$ but only in the sense of averages. This still allows the oscillations with uniformly positive amplitude to concentrate on sets of decreasing measure, and its possible for $\nu^{s(\cdot),(p,q)}(\dom;u)=+\infty$, for every $q>1$.
\begin{example}\label{E:StrictContain}
For this example, we use $\dom=(0,2)$ and $\bnd=\{0,2\}$, so assumptions (H1) and (H2) are obviously satisfied. Fix $p\ge1$ and $s_0\ge1/p$. We will construct $u\in L^\infty\cap\class{N}^{s_0,p}(\dom)$ such that $\nu^{s_0,(p,q)}(\dom;u)=+\infty$, for all $q>1$. Here, we are putting $s(\cdot)\equiv s_0$ on $\dom$. For each $j\in\nats$, define
\[
    a_j:=\frac{4^{-j\lp s_0-\frac{1}{p}\rp}}{5j^\frac{1}{p}\lp\ln(j+2)\rp^\frac{2}{p}},
    \quad
    E_j:=(4^{-j},a_j4^{-j}),
    \nd
    F_j:=[a_j4^{-j},4^{-j+1}],
\]
so $a_j\le\frac{6}{5}$ and $(0,1]=\bigcup_{j=1}^\infty E_j\cup F_j$. On the interval $(0,1]$, we define $u:=\sum_{j=1}^\infty\chi_{E_j}$ and extend $u$ to $(1,2)$ by putting $u(x):=u(2-x)$, for each $x\in(1,2)$. Since $u$ is symmetric, it is sufficient to show $\nu^{s_0,p}((0,1])<\infty$ and $\nu^{s_0,(p,q)}((0,1])=+\infty$, for $q>1$. We observe that, for each $j\in\nats$, if $x\in\wh{E}_j:=(4^{-j},2\cdot4^{-j})$, then
\[
    x<2\cdot4^{-j}\Longrightarrow\smfrac{1}{2}x<4^{-j}
    \nd
    x>4^{-j}\Longrightarrow\smfrac{3}{2}x>\smfrac{3}{2}4^{-j}\ge a_j4^{-j}
\]
Consequently, $E_j\subseteq\Psi(x)$. Furthermore, if $\Psi(x)\cap E_j\neq\emptyset$, then either
\[
    \smfrac{1}{2}x<a_j4^{-j}\Longrightarrow \smfrac{3}{2}x<3a_j4^{-j}\le 4^{-j+1}
    \:\text{ or }\:
    \smfrac{3}{2}x>4^{-j}\Longrightarrow\smfrac{1}{2}x>\smfrac{1}{3}4^{-j}\ge a_{j+1}4^{-j-1}.
\]
In either case, $\Psi(\vx)\cap E_{j-1}\cup E_{j+1}=\emptyset$. Set $\wh{F}_j:=[2\cdot4^{-j},4^{-j+1}]$. Now, if $x\in E_j\subseteq\wh{E}_j$, then $u(x)=1$ and
\[
    \slint{\Psi(x)}|u(y)-u(x)|^q\dd x=|\Psi(x)\bs E_j|
    \ge
    \lp4^{-j}-\smfrac{1}{2}x\rp+\lp\smfrac{3}{2}x-\smfrac{6}{5}4^{-j}\rp
    \ge
    \smfrac{3}{5}4^{-j}
    \ge\smfrac{1}{2}|E_j|.
\]
On the other hand, if $x\in \wh{E}_j\bs E_j$, then $u(x)=1$ and $\slint{\Psi(x)}|u(y)-u(x)|^q\dd x=|E_j|$.
We conclude that
\[
    \lc\begin{array}{ll}
        \smfrac{1}{2}|E_j|, & x\in\wh{E}_j,\\
        0, & x\in\wh{F}_j
    \end{array}\rt
    \le
    \slint{\Psi(x)}|u(y)-u(x)|^q\dd y
    \le
    |E_j|,
    \frl x\in(4^{-j},4^{-j+1}].
\]
Thus,
\begin{align*}
    \lc\begin{array}{ll}
        \lp\smfrac{1}{4}\rp^{s_0p+\frac{p}{q}}|E_j|^\frac{p}{q}4^{j\lp s_0p+\frac{p}{q}\rp}, & x\in\wh{E}_j\\
        0, & x\in\wh{F}_j
    \end{array}\rt
    \le&
    \lc\begin{array}{ll}
       \lp\smfrac{1}{2}|E_j|\rp^\frac{p}{q}
       d_{\partial\dom}(x)^{-s_0p}x^{-\frac{p}{q}}, & x\in\wh{E}_j\\
        0, & x\in\wh{F}_j
    \end{array}\rt\\
    \le&
    \lp\fslint{\Psi(x)}\frac{|u(y)-u(x)|^q}{d_{\partial\dom}(x)^{qs_0}}\dd y\rp^\frac{p}{q}\\
    \le&
    |E_j|^\frac{p}{q}d_{\partial\dom}(x)^{-s_0p}x^{-\frac{p}{q}}\\
    \le&
    |E_j|^\frac{p}{q}4^{j\lp s_0p+\frac{p}{q}\rp},\frl x\in (4^{-j},4^{-j+1}].
\end{align*}
Since $|E_j|=a_j4^{-j}$ and $|\wh{E}_j|=2\cdot4^{-j}$, we conclude that
\[
    \lp\frac{1}{8}\rp^{s_0p+\frac{p}{q}}\sum_{j=1}^\infty a_j^\frac{p}{q}4^{j(s_0p-1)}
    \le
    \sluint{0}{1}\lp\fslint{\Psi(x)}\frac{|u(y)-u(x)|^q}{d_{\partial\dom}(x)^{s_0}}
        \dd y\rp^\frac{p}{q}\dd x
    \le
    \sum_{j=1}^\infty a_j^\frac{p}{q}4^{j(s_0p-1)}.
\]
Plugging in the definition for $a_j$, we see that
\begin{align*}
    q=1
    \Longrightarrow&
    \nu^{s_0,(p,q)}(\dom;u)=\nu^{s_0,p}(\dom;u)
    \le
    \lp\frac{1}{5}\rp^p\sum_{j=1}^\infty\frac{1}{j\ln(j+2)^2}<\infty\\
    q>1
    \Longrightarrow&
    \nu^{s_0,(p,q)}(\dom;u)\ge
     \lp\frac{1}{8}\rp^{s_0p+\frac{p}{q}}
     \lp\frac{1}{5}\rp^\frac{p}{q}\sum_{j=1}^\infty\frac{4^{j\lp\frac{q-1}{q}\rp(s_0p-1)}}
        {j^\frac{1}{q}(\ln(j+2))^\frac{2}{q}}=+\infty
\end{align*}
\end{example}
\begin{remark}
If we modify the definition of $u$, on $(0,1]$, to $u:=\sum_{j=1}^\infty\frac{1}{\ln(j+1)}\chi_{E_j}$, then we find $\lim_{x\to0^+}u(x)=\lim_{x\to2^-}u(x)=0$ yet still $\nu^{s_0,(p,q)}=+\infty$, for all $q>1$.
\end{remark}

\subsection{Main Results}\label{SS:MainThms}
Given $u\in\class{N}^{s(\cdot),p}(\dom)$, we define $g:\dom\to\re$ by
\begin{equation}\label{D:gMeanFun}
    g(\vx;u):=\fslint{\Phi(\vx)}u(\vy)\dd\vy,
\end{equation}
where $\Phi(\vx):=\bll_{\frac{1}{6}d_{\partial\dom}(\vx)}(\vx)$. The function $g(\cdot;u)$ is continuous in $\dom$. Our first result identifies the trace of $u$ on $\bnd$ through a continuous extension of $g$ to $\bnd$ at $\haus^t$-a.e. point in $\bnd$.
\begin{theorem}\label{T:TraceExists}
Assume (H1) and (H2), and (H3$\,'$). Then there exists a linear operator $T:\class{N}^{s(\cdot),p}(\dom)\to L(\bnd)$ such that, for each $u\in\class{N}^{s(\cdot),p}(\dom)$ there exists a $\haus^t$-measurable set $\bnd'\subseteq\bnd$ such that
\begin{equation}\label{E:TraceLimit}
    \haus^t(\bnd\bs\bnd')=0
    \nd
    \lim_{\vx\to\ovx;\;\vx\in Q_\lambda(\ov{\vx})}g(\vx;u)=Tu(\ovx),
    \frl0<\lambda<1.
\end{equation}
Moreover, for every $\ovx\in\bnd'$ and $0<\beta<\lp\ul{\alpha}_0(\ovx)+t\rp/p$, there exists a $\delta_\beta=\delta_\beta(\ovx)>0$ with the following property: for each $0<\lambda<1$ there exists $C_{\beta,\lambda}(\ovx)<\infty$ such that
\begin{equation}\label{E:TraceHoldCont}
    |Tu(\ovx)-g(\vx;u)|
    \le
    C_{\beta,\lambda}(\ovx)
        \|\ovx-\vx\|^\beta,
    \frl\vx\in Q_{\lambda,\delta_\beta}(\ovx).
\end{equation}
\end{theorem}
\begin{remark}
If $u\in\cnt(\dom\cup\bnd)$, then $Tu(\vx)=u(\vx)$, for all $\vx\in\bnd$.
\end{remark}
Assuming a stronger oscillation constraint in a neighborhood of $\bnd$ and some regularity for $\bnd$, we may establish some differentiability and the Lebesgue point property for the trace operator provided by the previous theorem.
\begin{theorem}\label{T:TraceProp}
Assume (H1), (H2), and (H3$\,''$). Put $\beta_0:=(\ul{\alpha}_\bnd+t)/p>0$.
\begin{enumerate}[(a)]
    \item If $\bnd$ satisfies (H4$\,'$), then $T:\class{N}^{s(\cdot),p}(\dom)\to W^{\beta,p}(\bnd)$ is a continuous linear operator, for each $0<\beta<\beta_0$.
    \item If $\bnd=\partial\dom$ satisfies both (H4$\,'$) and (H4$\,''$) and $\ul{\alpha}_\bnd>n(\theta_\bnd(\ovx)-1)-t$, for all $\ovx\in\bnd$, then
\[
    \lim_{\rho\to0^+}\fslint{\dom_\rho(\ovx)}|Tu(\ovx)-u(\vx)|^p\dd\vx=0,
    \quad\text{for }\haus^t\text{-a.e. }\ovx\in\bnd.
\]
\end{enumerate}
\end{theorem}
\begin{remark}\label{R:ThmRmks}
\begin{enumerate}[(a)]
\item As mentioned in \cref{R:HypConn}, if $\dom$ is a $1$-sided NTA domain, then it satisfies (H1) and (H2), with $\bnd=\partial\dom$ and $\theta_\bnd\equiv1$. Assuming $\partial\dom$ is Ahlfors-regular, the hypotheses of \cref{T:TraceExists} and both parts of \cref{T:TraceProp} are satisfied if there exists an $s_0>(n-t)/p$ such that $s(\vx)\ge s_0$ in $\dom$.
\item~\label{R:PricklyExample} For a concrete example, consider $\dom$ and $\bnd$ as described in \cref{E:Domains}(\ref{E:Prickly}), with $\theta_0=2$, so $1<t<2$ and assumptions (H1), (H2), and (H4) are satisfied. Suppose that, for some $\delta_\bnd>0$ and constants $s_0<\infty$,
\[
    \alpha(s(\vx),2)\ge\alpha(s_0,2)
    =\lc\begin{array}{ll}
        2ps_0-p-4, & s_0\le(p+1)/p,\\
        ps_0-3, & s_0>(p+1)/p,
    \end{array}\rt
    \frl\vx\in U_\bnd,
\]
where $U_\bnd:=\bigcup_{\ovx\in\bnd}\dom_{\delta_\bnd}(\ovx)$. \cref{T:TraceExists} and \cref{T:TraceProp}(a) require $(p,s_0)$ to be in a region where $\alpha(s_0,2)>-t$ (shaded region in \cref{F:Alpha1}). For \cref{T:TraceProp}(b), we need $\alpha(s_0,2)>2-t$ (shaded region in \cref{F:Alpha2}).
\begin{figure}[h]
    \centering
    \hspace{0pt}
    \parbox[b]{2.73in}{
    \scalebox{.77}{\includegraphics[trim=150pt 565pt 200pt 125pt,clip]{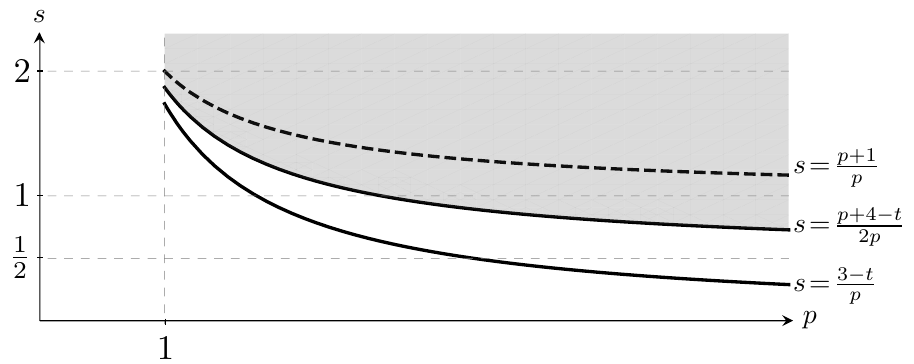}}
    \caption{Region where $\alpha(s,2)>-t$}
    \label{F:Alpha1}}
    \hspace{10pt}
    \parbox[b]{2.73in}{
    \scalebox{.77}{\includegraphics[trim=150pt 565pt 200pt 125pt,clip]{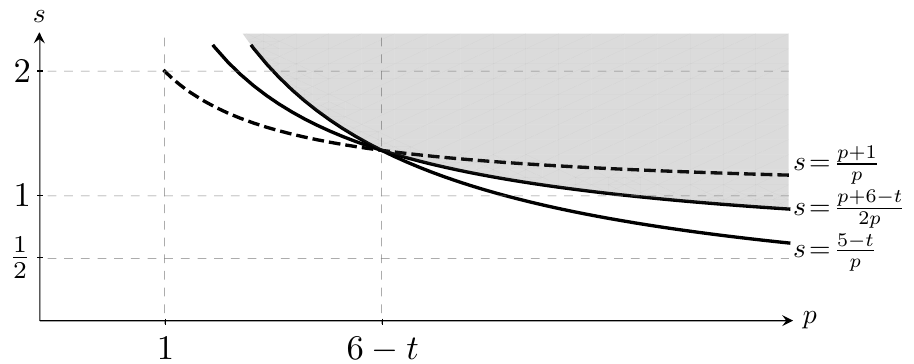}}
    \caption{Region where $\alpha(s,2)>2-t$}
    \label{F:Alpha2}
    }
\end{figure}\\
(Note: Generally, the curves do not have a common point of intersection.)
\item In Section~\ref{SS:Overview}, it was mentioned that $L^p_{\loc}(\dom)\subseteq\class{N}^{s(\cdot),p}(\dom)$. \cref{T:TraceProp}(b) implies that, for sufficiently small $\rho>0$, we find $u\in L^p(\dom_\rho(\ovx))$, for $\haus^t$-a.e. $\ovx\in\bnd$ (see \cref{R:TraceLebProp}(\ref{R:uIntegrability})). It is unclear, however, whether $L^p(\dom)\subseteq\class{N}^{s(\cdot),p}(\dom)$, even under the assumptions of \cref{T:TraceProp}(b).
\end{enumerate}
\end{remark}

\subsection{Organization of Paper}\label{SS:Organization}

In the next section, we establish some results needed for the main theorems. \cref{L:guMeansDiff}, in particular, provides the key bounds for the rate of change for $g$ within an approach region $Q_{\lambda}(\ovx)$. The main result in Section~\ref{S:Exists} is \cref{T:gGenResults}, which is a slight refinement of \cref{T:TraceExists}. As a corollary, we also show that the Lebesgue point property holds for the trace when the means are taking over $Q_{\lambda,\rho}(\ovx)$. Both parts of \cref{T:TraceProp} follow from the main results in Section~\ref{S:Props}. Part (a) is a consequence of \cref{T:uInWbp}, where a more precise connection between the fractional differentiability of $Tu$ and certain subsets of $\bnd$ is provided. The Lebesgue point property at points contained in a relative neighborhood of $\bnd$ is proved in \cref{T:LebProp}. The paper concludes with an Appendix, in Section~\ref{S:Appendix}, where the claims made in \cref{E:Domains}(\ref{E:Prickly}) are justified.

\section{Supporting Results}\label{S:SupResults}

For the remainder of the paper, fix $u\in\class{N}^{s(\cdot),p}(\dom)$ and put $\rho_\bnd:=\lp\frac{1}{3C_\bnd}\rp\delta_\bnd$. We use $c$ to denote a constant that may change from line to line but, unless otherwise indicated, is independent of the functions $s$ and $\theta_\bnd$ and the parameters $0<\eta\le\eta_0$, $0<\lambda\le\lambda_0$, and $\rho>0$. In particular, it may depend on $n$, $t$, $p$, $A_\bnd$, $C_\bnd$, and $\delta_\bnd$.

\begin{lemma}[Modified Giusti's Lemma]\label{L:ModGiusti}
Suppose that $\nu$ is an $\re$-valued function satisfying the following:
\begin{enumerate}[(i)]
\item The domain for $\nu$ includes all open sets in $\ren$;
\item $\nu$ is finite, nonnegative, and nondecreasing;
\item $\nu$ is countably superadditive; i.e. if $\{U_j\}_{j=1}^\infty$ are disjoint open subsets of $\ren$, then
\[
  \nu\lp\bigcup_{j=1}^\infty U_j\rp\ge\sum_{j=1}^\infty\nu(U_j);
\]
\end{enumerate}
Let $\tau\ge0$ be given, and set
\[
  S:=\lc\vx\in\ren
    :\limsup_{\rho\to0^+}\rho^{-\tau}\nu(\bll_\rho(\vx))=\infty\rc.
\]
Then $\dim_{\haus}(S)\le\tau$ and $\haus^\tau(S)=0$.
\end{lemma}
\begin{proof}
Recall that the $\haus^\tau$-outer measure $S\subset\ren$ is defined by
\[
    \haus^\tau(S)
    :=
    \lim_{\rho\to0^+} \haus^\tau_\rho(S),
\]
where
\begin{equation}\label{E:HausRho}
    \haus^\tau_\rho(S)
    :=
    \inf\lc\sum_{j=1}^\infty\diam(U_j)^\tau
        : \diam(U_j)<\rho\nd S\subseteq\bigcup_{j=1}^\infty U_j\rc.
\end{equation}
We will use Vitali's Covering Lemma. Let $0<K<\infty$ and $\rho>0$ be given. For each $\vx\in S_\tau$, we may select $0<r(\vx)<\rho$ such that
\[
    \nu(\bll_{r(\vx)}(\vx))>K\rho^\tau.
\]
Thus $S\subseteq\bigcup_{\vx\in S_\tau}\bll_{r(\vx)}(\vx)$. By Vitali's Covering Lemma, we may extract a countable sets $J$ and  $\lc\vx_j\rc_{j\in J}\subseteq S_\tau$ such that the sets
$\lc\bll_{r_j}(\vx_j)\rc_{j\in J}$, with $r_j:=r(\vx_j)$, are mutually disjoint and
\[
    \bigcup_{\vx\in S}\bll_{r(\vx)}(\vx)\subseteq\bigcup_{j\in J}\bll_{5r_j}(\vx_j).
\]
It follows that
\[
    \haus_\rho^\tau(S)
    \le
    \sum_{j\in J}\lp5r_j\rp^\tau
    =
    5^\tau\sum_{j\in J}r_j^\tau
    <\frac{5^\tau}{K}\sum_{j\in J}\nu\lp\bll_{r_j}(\vx_j)\rp
    \le
    \frac{5^\tau}{K}\nu\lp\bigcup_{j\in J}\bll_{r_j}(\vx_j)\rp
    \le
    \frac{5^\tau}{K}\nu(\ren).
\]
Since $0<K<\infty$ was arbitrary and $\nu(\ren)<\infty$, we conclude that $\haus_\rho^\tau(S)=0$. Taking the limit, as $\rho\to0^+$, yields the result.
\end{proof}
\begin{remark}
Our proof for \cref{L:ModGiusti} is a modification of one found, for example, in~\cite{DuzGasMin:04a} (see also~\cite{Giu:69a} for the original version). In~\cite{DuzGasMin:04a}, it is shown that the set
\[
    \lc\vx\in\dom:\limsup_{\rho\to0^+}\rho^{-\tau}\nu(\bll_r(\vx))>0\rc\subseteq S
\]
is a $\haus^\tau$-null set under the additional assumption that $\lim_{\vep\to0}\nu(U_\vep)=0$ whenever $\{U_\vep\}_{\vep>0}$ is a family of open sets satisfying $\lim_{\vep\to0}|U_\vep|=0$. (Note: A generalization of the spherical Hausdorff measure is actually considered in~\cite{DuzGasMin:04a} and $\haus^\tau$ is a special case.)
\end{remark}

We will also need the following fact, which follows from Besicovitch's covering theorem and the fact that there is a constant $c=c(n)$ such that packing number of the unit ball with balls of radius $r>0$ is bounded by $cr^{-n}$.
\begin{lemma}\label{L:BoundedOverlap}
There exists a constant $c=c(n)$ with the following property: for any $E\subseteq\ren$ and $r>0$ and $\Lambda\ge1$, there exists a countable set $I$ and set of points $\{\vx_i\}_{i\in I}\subseteq E$ such that $E\subseteq\bigcup_{i\in I}\bll_r(\vx_i)$ and $\sup_{\vx\in\ren}\sum_{i\in I}\chi_{\bll_{\Lambda r}(\vx_i)}(\vx)\le c\Lambda^n$.
\end{lemma}

Recall that $\bnd_\rho(\ovx)=\bnd\cap\bll_\rho(\ovx)$, for each $\rho>0$ and $\ovx\in\bnd$. The following result follows from the argument for statement (2) of Theorem 1.1 in~\cite{Hu:03a}. A proof is included for the sake of completeness.
\begin{theorem}
    Suppose that $\bnd$ satisfies (H4$''$); i.e. is lower Ahlfors-regular; and that $v\in W^{\beta,p}(\bnd)$, for some $\beta>0$. Define $v_{\beta}:\bnd\to\reo$ by
    \[
        v_{\beta}(\ovx):=\lp\slint{\bnd_1(\ovx)}
            \frac{|v(\ovz)-v(\ovx)|^p}{\|\ovz-\ovx\|_{\ren}^{t+\beta p}}
            \dd\haus^t(\ovz)\rp^\frac{1}{p}.
    \]
    Then $v_\beta\in L^p(\bnd)$ and
    \begin{equation}\label{E:vHoldCont}
        |v(\ovy)-v(\ovx)|
        \le
        c\|\ovy-\ovx\|_{\ren}^{\beta}\lp v_\beta(\ovy)+v_\beta(\ovx)\rp,
        \frl\ovx,\ovy\in\bnd.
    \end{equation}
\end{theorem}
\begin{proof}
From the definition of $W^{p,\beta}(\bnd)$, it is immediate that $v_\beta\in L^p(\bnd)$. To verify~\eqref{E:vHoldCont}, it is sufficient to consider $\ovx,\ovy\in\bnd$ such that $\rho:=\|\ovx-\ovy\|_{\ren}\le\frac{1}{2}$. Since $\bll_\rho(\ovy)\subseteq\bll_{2\rho}(\ovx)$, by assumption (H4$''$), we have
\begin{align*}
    |v(\ovy)-v(\ovx)|^p
    \le&
    c\fslint{\bnd_\rho(\ovy)}|v(\ovz)-v(\ovy)|^p\dd\haus^t(\ovz)
    +c\fslint{\bnd_{2\rho}(\ovx)}|v(\ovz)-v(\ovx)|^p\dd\haus^t(\ovz)\\
    \le&
    c\rho^{\beta p}\ns\slint{\bnd_\rho(\ovy)}\ns
            \frac{|v(\ovz)-v(\ovy)|^p}{\|\ovz-\ovy\|_{\ren}^{t+\beta p}}
            \dd\haus^t(\ovz)
    +
    c\rho^{\beta p}\ns\slint{\bnd_{2\rho}(\ovx)}\ns
            \frac{|v(\ovz)-v(\ovx)|^p}{\|\ovz-\ovx\|_{\ren}^{t+\beta p}}
            \dd\haus^t(\ovz)\\
    \le&
    c\|\ovy-\ovx\|_{\ren}^p\lp v_\beta(\ovy)+v_\beta(\ovx)\rp^p.
\end{align*}
\end{proof}
Our proof of \cref{T:TraceProp}(b) requires the following
\begin{corollary}\label{C:LebPoints}
    If $\bnd$ satisfies (H4$''$) and $v\in W^{\beta,p}(\bnd)$, then for any $0<\beta'<\beta$
\[
    \lim_{\rho\to0^+}\rho^{-p\beta'}\fslint{\bnd\cap\bll_\rho(\ovx_0)}
        |v(\ovx)-v(\ovx_0)|^p\dd\haus^t(\ovx)=0,
    \quad\text{for $\haus^t$-a.e. }\ovx_0\in\bnd.
\]
\end{corollary}
\begin{proof}
The previous theorem implies $v_\beta\in L^p(\bnd)$, so $|v_\beta(\ovx_0)|<\infty$, for $\haus^t$-a.e. $\ovx_0\in\bnd$. Given an open set $U\subseteq\re^n$, define
\[
    \nu(U):=\slint{\bnd\cap U}v_\beta(\ovx)^p\dd\haus^t(\ovx).
\]
Then $\nu$ satisfies the hypotheses of \cref{L:ModGiusti}. Under assumption (H4$''$), we deduce that
\[
    \lim_{\rho\to0^+}\fslint{\bnd_\rho(\ovx_0)}v_\beta(\ovx)^p\dd\haus^t(\ovx)
    \le
    c\rho^{-t}\nu(\bll_\rho(\ovx_0))<\infty,
    \quad\text{ for $\haus^t$-a.e. }\ovx_0\in\bnd.
\]
The result now follows from~\eqref{E:vHoldCont}.
\end{proof}

The next lemma collects some statements that are straightforward consequences of the definitions and assumptions.
\begin{lemma}\label{L:Hypotheses}
Let $\ovx\in\bnd$ and $0<\eta,\lambda<1\le\theta$ be given.
\begin{enumerate}[(a)]
\item\label{L:Open} For each $\delta>0$, the set $Q^\theta_{\lambda,\delta}(\ovx)$ is open.
\item\label{L:Corkscrew} Clearly $Q^{\theta}_{\lambda}(\ovx)\subseteq Q^{\theta'}_{\lambda'}(\ovx)$, for any $0<\lambda'\le\lambda<1\le\theta\le\theta'<\infty$. Thus, if $Q^\theta_{\lambda,\delta}(\ovx)$ is an $(\eta,\theta)$-corkscrew region, then $Q^{\theta'}_{\lambda',\delta}(\ovx)$ is an $(\eta',\theta')$-corkscrew region, for each $0<\eta'\le\eta$.
\item Let $0<\rho,\rho'<1$ be given. If $\ovx'\in\bnd$ satisfies $\|\ovx-\ovx'\|_{\ren}<\rho'$ and $\theta'\ge\theta$, then $Q^\theta_{\lambda,\rho}\bs\bll_{\eta\rho}(\ovx)\subseteq Q^{\theta'}_{\lambda'}(\ovx')$, with $\lambda':=\lambda\eta\rho/(\rho+\rho')$.
\item\label{L:ThinCorkscrew} Let $\delta_0>0$ and $C\ge1$ be given. Suppose that $Q^\theta_{\lambda}(\ovx)$ is both $(C,\theta)$-connected to $\ovx$ and an $(\eta,\theta)$-corkscrew region radius $\delta_0$. Then $Q^\theta_\lambda(\ovx)$ is an $(\eta',\theta)$-corkscrew region with radius $\eta\delta_0$, for every $0<\eta'<1$.
\item\label{L:CorkscrewVol} Assume (H1), and put $R_0:=\sup_{\vx\in Q^{\theta_\bnd(\ovx)}_{\lambda_0,\delta}(\ovx)}d_{\partial\dom}(\vx)$.
Then, $R_0\ge(\eta_0\lambda_0\delta)^{\theta_\bnd(\ovx)}$, and there exists a dimensional constant $c>0$ such that, for each $0<\lambda<\lambda_0$,
\[
    |Q^{\theta_\bnd(\ovx)}_{\lambda,\delta}(\ovx)|
    \ge
    c(1-\lambda/\lambda_0)^{n\theta_\bnd(\ovx)}R_0^n
    \nd
    |\dom_\rho(\ovx)|\ge cR_0^n.
\]
\end{enumerate}
\end{lemma}
\begin{proof}
Parts (a)--(c) are direct consequences of the definitions.

For part (d), by the definition of an $(\eta,\theta)$-corkscrew region, for each $j\in\whls$, there exists $\vx_j\in Q^\theta_{\lambda,\eta^j\delta_0}\bs\bll_{\eta^{j+1}\delta_0}(\ovx)$. The connectedness assumption implies there is a path of points in $Q^\theta_\lambda(\ovx)$ joining $\vx_j$ and $\vx_{j+1}$, and this yields the claim.

For part (e), assumption (H1) implies there exists an $\vx\in Q^{\theta_\bnd(\ovx)}_{\lambda_0,\delta}(\ovx)$ and the lower bound for $R_0$. Put $\rho:=\|\ovx-\vx\|_{\ren}$ and $r:=d_{\partial\dom}(\ovx)$, so $r\ge(\lambda_0\rho)^{\theta_\bnd(\ovx)}$. Since $\bll_r(\vx)\subseteq\dom$ and $\delta\ge\rho\ge r$, we find that $\bll_r(\vx)\cap\dom_{\delta}(\ovx)$ contains a set congruent to $\bll_{\rho}(\ve{0})\cap\bll_{r}(\ve{e}\rho)$, with $\ve{e}\in\ren$ a unit vector. Thus, there is dimensional constant $c>0$ such that
\[
    |\bll_{\sigma r}(\vx)\cap\dom_{\delta}(\ovx)|\ge c\sigma^nr^n,
    \frl 0<\sigma\le 1.
\]
For each $0<\sigma<1$ and $\vy\in\bll_{\sigma r}(\vx)\cap\dom_{\delta}(\ovx)$, we have
\[
    d_{\partial\dom}(\vy)
    \ge
    d_{\partial\dom}(\vx)-\sigma r
    =(1-\sigma)r
    \nd
    \|\ovx-\vy\|_{\ren}\le\delta.
\]
We may select $0<\sigma<1$ so that $\lambda=(1-\sigma)^{1/\theta_\bnd(\ovx)}\lambda_0$. It follows that $\bll_{\sigma r}(\vx)\cap\dom_{\delta}(\ovx)\subseteq Q^{\theta_\bnd(\ovx)}_{\lambda,\delta}(\ovx)$. Hence
\[
    |Q^{\theta_\bnd(\ovx)}_{\lambda,\delta}(\ovx)|\ge c\sigma^nr^n
    \ge c\lp(\lambda_0-\lambda)\rho\rp^{n\theta_\bnd(\ovx)}.
\]
Since $\vx\in Q^{\theta_\bnd(\ovx)}_{\lambda_0,\delta}(\ovx)$ was arbitrary, we deduce that
\[
    |Q^{\theta_\bnd(\ovx)}_{\lambda,\delta}(\ovx)|\ge c\sigma^nR_0^n
    =c\lp1-\lambda/\lambda_0\rp^{n\theta_\bnd(\ovx)}R_0^n.
\]
This also yields $|\dom_\rho(\ovx)|=\lim_{\lambda\to0^+}|Q^{\theta_\bnd(\ovx)}_{\lambda,\delta}|
    \ge cR_0^n$.
\end{proof}

\begin{lemma}\label{L:guMeansDiff}
Assume (H2). Let $\ovx\in\bnd$ be given. Then for each $0<\rho\le\rho_\bnd$, $0<\eta\le\eta_0$, $0<\lambda\le\lambda_0$ and $\vx,\vx'\in Q_{\lambda,\rho}\bs\bll_{\eta\rho}(\ovx)$, we have
\[
    |g(\vx;u)-g(\vx';u)|^p
    \le
    c\vep_{\lambda'}^{-n-p}\delta^{\alpha_\delta(\ovx)}
        \nu^{s(\cdot),p}\lp\bll_\delta(\ovx)\rp
\]
Here $\lambda':=\eta\lambda$ and $\delta:=3C_\bnd\rho$.
\end{lemma}
\begin{proof}
Put $\theta:=\theta_\bnd(\ovx)$ and $r:=\rho^{\theta}$, so $r\le\rho$. We may assume $\vx\neq\vx'$. From the definition of $Q_{\lambda,\rho}(\ovx)$, we have
\[
    \rho>\|\ovx-\vx\|_{\ren},\|\ovx-\vx'\|_{\ren}\ge\eta\rho,
    \qqquad
    \rho>d_{\partial\dom}(\vx),d_{\partial\dom}(\vx')>(\eta\lambda)^{\theta}r.
\]
By hypothesis (H2), there exists a path $\ve{\gamma}:[0,1]\to\dom_{\delta_\bnd}(\ovx)$, between $\vx$ and $\vx'$, satisfying~\eqref{H:PathProp}. We assume that $\gamma$ is injective. Put $\tau_0:=0$ and $\vx_0:=\vx$. If $\haus^1(\ve{\gamma}([0,1]))>\frac{1}{12}d_{\partial\dom}(\vx_0)$, then we select $\tau_1\in(0,1)$ so that $\haus^1(\ve{\gamma}([\tau_0,\tau_1]))=\frac{1}{12}d_{\partial\dom}(\vx_0)$ and define $\vx_1:=\ve{\gamma}(\tau_1)$. We continue iteratively: if $\haus^1(\ve{\gamma}([\tau_j,1]))>\frac{1}{12}d_{\partial\dom}(\vx_j)$, then we we select $\tau_{j+1}\in(\tau_j,1)$ so that $\haus^1(\ve{\gamma}([\tau_j,\tau_{j+1}]))=\frac{1}{12}d_{\partial\dom}(\vx_j)$ and define $\vx_{j+1}:=\ve{\gamma}(\tau_{j+1})$. Since $d_{\partial\dom}(\ve{\gamma}(\tau))\ge\vep_{\lambda'}r$, for all $\tau\in[0,1]$, there exists an $N\in\nats$ such that $\haus^1(\ve{\gamma}([\tau_{N-1},1]))<\frac{1}{12}d_{\partial\dom}(\vx_{N-1})$. Indeed, since $\haus^1(\ve{\gamma}([0,1]))\le C_\bnd\rho$, we have the bound
\begin{equation}\label{E:NBnd}
    N\le 1+\frac{C_\bnd\rho}{\frac{1}{12}\vep_{\lambda'}r}
    \le c\lp\frac{C_\bnd}{\vep_{\lambda'}}\rp\rho^{1-\theta}.
\end{equation}
Put $\vx_N:=\vx'$. To summarize, the finite sequence $\{\vx_j\}_{j=0}^N\subset\ve{\gamma}([0,1])$ has the following properties:
\begin{itemize}
    \item $\vx_0=\vx$ and $\vx_N=\vx'$;
    \item $\|\vx_j-\vx_{j-1}\|_{\ren}\le\frac{1}{12}d_{\partial\dom}(\vx_{j-1})$, for each $j=1,\dots,N$;
    \item $\vep_{\lambda'}r\le d_{\partial\dom}(\vx_j)\le\|\ovx-\vx_j\|_{\ren}<\delta_\bnd$, for each $j=0,1,\dots,N$.
\end{itemize}
Before continuing, we note that
\begin{equation}\label{E:xjDistBnd}
\nonumber
    d_{\partial\dom}(\vx_j)
    \le
    \min\lc d_{\partial\dom}(\vx_0)+\|\vx_0-\vx_j\|_{\ren},d_{\partial\dom}(\vx_N)+\|\vx_N-\vx_j\|_{\ren}\rc
    \le
    (1+\smfrac{1}{2}C_\bnd)\rho\le\smfrac{1}{2}\delta.
\end{equation}
Put $L:=\frac{1}{12}\sum_{j=0}^{N-1}d_{\partial\dom}(\vx_j)$. Then
\begin{equation}\label{E:LBnd}
    \|\vx-\vx'\|_{\ren}\le\haus^1(\ve{\gamma}([0,1]))\\
    \le L\le\haus^1(\ve{\gamma}([0,1]))+\smfrac{1}{12}d_{\partial\dom}(\vx_{N-1})
    \le \lp\smfrac{1}{12}+\smfrac{25}{24}C_\bnd\rp\rho\le\smfrac{9}{8}C_\bnd\rho.
\end{equation}
For each $j=1,\dots,N$, define $\zeta_j:=\lp\frac{1}{12L}\rp d_{\partial\dom}(\vx_{j-1})$, so $\sum_{j=1}^N\zeta_j=1$.

We proceed now to the main part of the proof. The convexity of $x\to|x|^p$ allows us to write
\begin{equation}\label{E:gDiffIneq1}
    |g(\vx)-g(\vx')|^p
    \le
    \sum_{j=1}^N\zeta_j^{1-p}
        \underbrace{\left|\fslint{\Phi'(\vx_j)}u(\vy_j)\dd\vy_j
            -\fslint{\Phi'(\vx_{j-1})}u(\vy_{j-1})\dd\vy_{j-1}\right|^p}_{=:I_j}.
\end{equation}
We observe that, for each $j=0,1,\dots,N$ and $\vy_j\in\Phi(\vx_j)$,
\begin{equation}\label{E:yjDistBnd}
    d_{\partial\dom}(\vy_j)<\lp1+\smfrac{1}{6}\rp d_{\partial\dom}(\vx_j)
    \le\smfrac{7}{12}\delta<\delta_\bnd.
\end{equation}
With $j=1,2,\dots,N$, let $\vy_{j-1}\in\Phi(\vx_{j-1})$ be given. Then
\begin{align*}
    & \|\vy_{j-1}-\vx_{j-1}\|_{\ren}<\smfrac{1}{6}d_{\partial\dom}(\vx_{j-1})\\
    \Longrightarrow &
    d_{\partial\dom}(\vy_{j-1})
    \ge
    d_{\partial\dom}(\vx_{j-1})-\|\vy_{j-1}-\vx_{j-1}\|_{\ren}>\lp1-\smfrac{1}{12}\rp d_{\partial\dom}(\vx_{j-1})\\
    \Longrightarrow &
    d_{\partial\dom}(\vx_{j-1})<\smfrac{12}{11}d_{\partial\dom}(\vy_{j-1}).
\end{align*}
Thus, for each $\vy_{j-1}\in\Phi(\vx_{j-1})$, we have
\begin{align}
\nonumber
    \|\vy_j-\vy_{j-1}\|_{\ren}
    &\le
    \|\vy_j-\vx_j\|_{\ren}+\|\vx_j-\vx_{j-1}\|_{\ren}+\|\vx_{j-1}-\vy_{j-1}\|_{\ren}\\
\nonumber
    &<
    \smfrac{1}{6} d_{\partial\dom}(\vx_j)+\smfrac{1}{12}d_{\partial\dom}(\vx_{j-1})
        +\smfrac{1}{12}d_{\partial\dom}(\vx_{j-1})\\
\nonumber
    &\le
    \smfrac{1}{6}\lp d_{\partial\dom}(\vx_{j-1})+\|\vx_j-\vx_{j-1}\|_{\ren}\rp
        +\smfrac{1}{6}d_{\partial\dom}(\vx_{j-1})\\
\nonumber
    &\le
    \lp\smfrac{1}{6}+\smfrac{1}{72}+\smfrac{1}{6}\rp d_{\partial\dom}(\vx_{j-1})
    =\smfrac{25}{72}d_{\partial\dom}(\vx_{j-1})\\
\label{E:Phi'InPsi}
    &<\smfrac{1}{2}d_{\partial\dom}(\vy_{j-1}).
\end{align}
Consequently, $\Phi(\vx_j)\subseteq\Psi(\vy_{j-1})\subseteq\dom$, for all $\vy_{j-1}\in\Phi(\vx_{j-1})$.
Additionally,
\begin{align*}
    &d_{\partial\dom}(\vx_{j-1})\le d_{\partial\dom}(\vx_j)+\|\vx_j-\vx_{j-1}\|_{\ren}
    \le d_{\partial\dom}(\vx_j)+\smfrac{1}{12}d_{\partial\dom}(\vx_{j-1})\\
    \Longrightarrow&
    d_{\partial\dom}(\vx_{j-1})\le\smfrac{12}{11}d_{\partial\dom}(\vx_j).
\end{align*}
It follows that
\[
    d_{\partial\dom}(\vy_{j-1})
    <
    \lp1+\smfrac{1}{6}\rp d_{\partial\dom}(\vx_{j-1})
    \le\smfrac{14}{11}d_{\partial\dom}(\vx_j)
    \Longrightarrow
    \frac{|\Psi(\vy_{j-1})|}{|\Phi(\vx_j)|}
    \le
    4^n.
\]
Using this,~\eqref{E:Phi'InPsi} and Jensen's inequality yields
\begin{align*}
    I_j
    =&
    \left|\fslint{\Phi(\vx_{j-1})}\flint{\Phi'(\vx_j)}[u(\vy_j)
        -u(\vy_{j-1})]\dd\vy_j\dd\vy_{j-1}\right|^p\\
    \le&
    \fslint{\Phi(\vx_{j-1})}\lp\frac{|\Psi(\vy_{j-1})|}{|\Phi'(\vx_j)|}\rp^p
        \lp\fslint{\Psi(\vy_{j-1})}|u(\vy_j)-u(\vy_{j-1})|\dd\vy_j\rp^p\dd\vy_{j-1}\\
    \le&
    c
    \fslint{\Phi(\vx_{j-1})}
        \lp\fslint{\Psi(\vy_{j-1})}|u(\vy_j)-u(\vy_{j-1})|\dd\vy_j\rp^p\dd\vy_{j-1}.
\end{align*}
Returning to~\eqref{E:gDiffIneq1}, after reindexing, we have established
\begin{align}
\nonumber
    |g(\vx)-g(\vx')|^p
    \le&
    c\sum_{j=0}^{N-1}\zeta_{j+1}^{1-p}
        \lp\frac{1}{d_{\partial\dom}(\vx_j)}\rp^n\lint{\Phi(\vx_j)}
        \lp\flint{\Psi(\vy)}|u(\vy')-u(\vy)|\dd\vy'\rp^{\!p}
        \ns\dd\vy\\
\nonumber
    \le&
    cL^{p-1}\sum_{j=0}^{N-1}
        \slint{\Phi(\vx_j)}
            \frac{\lp d_{\partial\dom}(\vy)\rp^{ps(\vy)}}{d_{\partial\dom}(\vx_j)^{n+p-1}}
        \lp\fslint{\Psi(\vy)}\frac{|u(\vy')-u(\vy)|}{d_{\partial\dom}(\vy)^{s(\vy)}}\dd\vy'\rp^{\!p}
        \ns\dd\vy\\
\label{E:gDiffIneq2}
    \le&
    c\rho^{p-1}\sum_{j=0}^{N-1}
        \slint{\Phi(\vx_j)}\ns\lp d_{\partial\dom}(\vx_j)\rp^{ps(\vy)-n-p+1}\!
        \lp\fslint{\Psi(\vy)}\ns\frac{|u(\vy')-u(\vy)|}{d_{\partial\dom}(\vy)^{s(\vy)}}\dd\vy'\!\rp^{\!p}
        \ns\dd\vy.
\end{align}
The last line is a consequence of $d_{\partial\dom}(\vy)<\frac{7}{6}d_{\partial\dom}(\vx_j)\le\frac{7}{12}\delta<1$ and the bounds in~\eqref{E:LBnd}. Next, we recall~\eqref{E:xjDistBnd} to see that, for any $\vy\in\Phi(\vx_j)$,
\begin{align*}
    \|\ovx-\vy\|_{\ren}
    \le&
    \min\lc\|\ovx-\vx\|_{\ren}\!+\!\|\vx-\vx_j\|_{\ren},\|\ovx-\vx'\|_{\ren}\!+\!\|\vx'-\vx_j\|_{\ren}\rc
        +\|\vx_j-\vy\|_{\ren}\\
    <&
    \rho+\smfrac{1}{2}C_\bnd\rho+\smfrac{1}{6}d_{\partial\dom}(\vx_j)
    <3C_\bnd\rho=\delta.
\end{align*}
Thus $\Phi(\vx_j)\subseteq\bll_\delta(\ovx)$ and $s(\vy)\ge\ul{s}_\delta(\ovx)$, for all $\vy\in\Phi(\vx_j)$. If $p\ul{s}_\delta(\ovx)\le n+p-1$, then we may incorporate the bounds $d_{\partial\dom}(\vx_j)\ge\vep_{\lambda'}\rho^\theta$ and~\eqref{E:NBnd} into~\eqref{E:gDiffIneq2} to write
\begin{align*}
    |g(\vx)-g(\vx')|^p
    \le&
    c\vep_{\lambda'}^{p\ul{s}_\delta(\ovx)-n-p+1}\rho^{p-1}\delta^{\lp p\ul{s}_\delta(\ovx)-n-p+1\rp\theta}\sum_{j=0}^{N-1}
        \nu^{s(\cdot),p}\lp\Phi'(\vx_j)\rp\\
    \le&
    c\vep_{\lambda'}^{-n-p+1}N\delta^{(p-1)(1-\theta)}\delta^{\lp p\ul{s}_\delta(\ovx)-n\rp\theta}
        \nu^{s(\cdot),p}\lp\bll_\delta(\ovx)\rp\\
    \le&
    c\vep_{\lambda'}^{-n-p}\delta^{p(1-\theta)}\delta^{\lp p\ul{s}_\delta(\ovx)-n\rp\theta}
        \nu^{s(\cdot),p}\lp\bll_\delta(\ovx)\rp.
\end{align*}
On the other hand, if $p\ul{s}_\delta(\ovx)> n+p-1$, then we may use~\eqref{E:xjDistBnd} to similarly obtain
\[
    |g(\vx)-g(\vx')|^p
    \le
    cN\delta^{p\ul{s}(\ovx)-n}
        \nu^{s(\cdot),p}\lp\bll_\delta(\ovx)\rp
    \le
    c\vep_{\lambda'}^{-1}\delta^{1-\theta}\delta^{p\ul{s}(\ovx)-n}
        \nu^{s(\cdot),p}\lp\bll_\delta(\ovx)\rp.
\]
\end{proof}

\begin{lemma}\label{L:gx'gxDiffBnd}
Assume (H1) and (H2). Let $\ovx\in\bnd$, $0<\eta\le\eta_0$, and $0<\lambda\le\lambda_0$ be given. For each $j\in\whls$, define $\rho_j:=\eta^j\rho_\bnd$, $\delta_j:=\eta^j\delta_\bnd$, and $Q'_j:=Q_{\lambda,\rho_j}\bs\bll_{\rho_{j+1}}(\ovx)$. Then for any $0\le k< k'<\infty$, we have
\[
    |g(\vx')-g(\vx)|\le
    c\vep_{\lambda'}^{-\frac{n+p}{p}}\sum_{j=k}^{k'-1}
        \lp\delta_j^{\ul{\alpha}_{\delta_j}(\ovx)}
            \nu^{s(\cdot),p}\lp\bll_{\delta_j}(\ovx)\rp\rp^\frac{1}{p},
    \text{ for all }\vx\in Q'_k\text{ and }\vx'\in Q'_{k'}.
\]
Here $\lambda'=\eta^2\lambda$.
\end{lemma}
\begin{proof}
By \cref{L:Hypotheses}(\ref{L:Corkscrew}), hypothesis (H1) implies $Q'_j\neq\emptyset$, for each $j\in\whls$. We may therefore select a sequence $\{\vx_j\}_{j=0}^\infty\subseteq Q_\lambda(\ovx)$ with the following properties:
\[
    \vx_j\in Q'_j,\quad \vx_k=\vx\nd\vx_{k'}=\vx'.
\]
We notice that $\vx_j,\vx_{j+1}\in Q'_j\cup Q'_{j+1}=Q_{\lambda,\rho_j}\bs\bll_{\eta^2\rho_j}(\ovx)$. \cref{L:guMeansDiff} implies
\[
    |g(\vx_{j+1})-g(\vx_j)|^p\le c\vep_{\lambda'}^{-n-p}\delta_j^{\ul{\alpha}_{\delta_j}(\ovx)}
        \nu^{s(\cdot),p}(\bll_{\delta_j}(\ovx).
\]
Taking the $\supth{p}$-root and summing for $j=k,\dots,k'-1$ yields the result.
\end{proof}

\section{Existence of a Trace}\label{S:Exists}

It is clear that $g$ is continuous in $\dom$. We next show that $g$ can be continuously extended to those points $\ovx\in\bnd$ where $\ul{\alpha}_0(\ovx)>t$. Where they exist, the values of this extension on $\bnd$ can be used to define the trace of $u$. To this end, we observe that
\[
    \vx\mapsto\lp\fslint{\Psi(\vx)}\frac{|u(\vy)-u(\vx)|}{d_{\partial\dom}(\vx)^{s(\vx)}}\rp^p
    \in L^1(\dom),
\]
so $\nu^{s(\cdot),p}(\cdot;u)$ is a measure that is absolutely continuous with respect to the Lebesgue measure. It therefore satisfies the hypotheses of \cref{L:ModGiusti}. Consequently, for each $\tau\ge0$, the set
\[
  S_\tau:=\lc\vx\in\ren
    :\limsup_{\rho\to0^+}\rho^{-\tau}
      \nu^{s(\cdot),p}(\bll_\rho(\vx);u)=\infty\rc
\]
has Hausdorff dimension of at most $\tau$ and $\haus^\tau(S_\tau)=0$. Recall that we are working under the assumption that $\bnd$ has Hausdorff dimension $t$. Thus, if $0\le\tau<t$, then
\begin{equation}\label{E:limsupNuBound}
  \limsup_{\rho\to0^+}\rho^{-\tau}
      \nu^{s(\cdot),p}(\bll_\rho(\ovx))<\infty,
  \quad\text{for $\haus^t$-a.e. }\ovx\in\bnd.
\end{equation}
For each $\omega\in\re$ and $\tau\ge0$, set
\[
    A_{\omega,0}:=\lc\ovx\in\bnd:\ul{\alpha}_0(\ovx)+\omega\ge0\rc
    \in\class{B}(\partial\dom)
    \nd
    \bnd_\tau:=\bnd\bs S_\tau\in\class{B}(\partial\dom),
\]
and define $M_\tau:\bnd\to[0,\infty]$ by
\[
    M_\tau(\ovx):=\sup_{0<\delta\le\delta_\bnd}\rho^{-\tau}
        \nu^{s(\cdot),p}\lp\bll_\rho(\ovx)\rp<\infty.
\]
Clearly $M_\tau$ is Borel-measurable and $\bnd_{\tau}\supseteq\bnd_{\tau'}$ and $M_{\tau}(\ovx)\le M_{\tau'}(\ovx)$, whenever $0\le\tau\le\tau'$. Moreover, from the discussion above $\haus^t(\bnd\bs\bnd_\tau)=0$, for all $0\le\tau\le t$.

The following theorem identifies points $\ovx\in\bnd$ where we can use \cref{L:gx'gxDiffBnd} to establish that there is an $\re$-valued continuous extension of $g$ to $\ovx$.
\begin{theorem}\label{T:gGenResults}
Assume (H1) and (H2). Let $\tau_0\ge 0$ be given. With $\omega_0<\tau_0$, suppose that $\ovx\in A_{\omega_0,0}\cap\bnd_{\tau_0}$. Then, there exists a $g(\ovx)\in\re$ such that, for each $0<\lambda<1$,
\begin{equation}\label{E:gLimit}
    \lim_{\vx\to\ov{\vx};\;\vx\in Q_\lambda(\ov{\vx})}g(\vx)=g(\ovx).
\end{equation}
Moreover, for each $0<\beta<(\tau_0-\omega_0)/p$, there exists $\delta_\beta=\delta_\beta(\ovx)$ such that
\begin{equation}\label{E:gHoldCont}
    |g(\ovx)-g(\vx)|
    \le
    C M_{\tau_0}(\ovx)^\frac{1}{p}
        \|\ovx-\vx\|^\beta,
    \frl\vx\in Q_{\lambda,\delta_\beta}(\ovx).
\end{equation}
Here
\[
    C=C(\eta_0,\lambda,\beta)
    :=
    c\vep_{\lambda'}^{-\frac{n+p}{p}}\eta_0^{-\beta}\lp1-\eta_0^\beta\rp^{-1}
    \nd
    \lambda':=\min\{\eta_0^2\lambda_0,\eta_0^2\lambda\}.
\]
\end{theorem}
\begin{proof}
First we identify a candidate for $g(\ov{\vx})$. We then establish~\eqref{E:gHoldCont}, and~\eqref{E:gLimit} immediately follows.

Put $\lambda_0':=\lambda_0\eta_0^2$. For each $j\in\whls$, also define $\rho_j:=\eta_0^j\rho_\bnd$, $\delta_j:=\eta_0^j\delta_\bnd=3C_\bnd\rho_j$, and $Q'_j:=Q_{\lambda_0,\rho_j}\bs\bll_{\rho_{j+1}}(\ovx)$. As in \cref{L:gx'gxDiffBnd}, by hypotheses (H1), we may select $\vx_j\in Q'_j$. Put $\beta_0:=(\tau_0-\omega_0)/p$. Since $\ovx\in A_{\omega_0,0}$, we have $0<\beta_0\le(\ul{\alpha}_0(\ovx)+\tau_0)/p$. Since $\ovx\in\bnd_{\tau_0}$, we also have $M_{\tau_0}(\ovx)<\infty$. It was noted before that $\ul{\alpha}_\delta(\ovx)$ increases as $\delta\searrow0^+$. We deduce that there exists $k_0\in\whls$ such that $\ul{\alpha}_{\delta_j}(\ovx)+\frac{1}{2}(\tau_0+\omega_0)\ge0$, for every $j\ge k_0$. Thus, given $k'\ge k_0$, \cref{L:gx'gxDiffBnd} yields
\begin{align*}
    \sum_{j=k_0}^{k'}|g(\vx_{j+1})-g(\vx_j)|
    &\le
    c\vep_{\lambda'_0}^{-\frac{n+p}{p}}\sum_{j=k_0}^{\infty}
        \lp\delta_j^{\alpha_{\delta_j(\ovx)}}
            \nu^{s(\cdot),p}\lp\bll_{\delta_j}(\ovx)\rp\rp^\frac{1}{p}\\
    &=
    c\vep_{\lambda'_0}^{-\frac{n+p}{p}}\sum_{j=k_0}^{\infty}
        \lp\delta_j
            ^{\lp\alpha_{\delta_j}(\ovx)+\frac{\tau_0+\omega_0}{2}\rp
                +\lp\tau_0-\smfrac{\tau_0+\omega_0}{2}\rp}
            \delta_j^{-\tau_0}\nu^{s(\cdot),p}\lp\bll_{\delta_j}(\ovx)
            \rp\rp^\frac{1}{p}\\
    &\le
    c\lp\vep_{\lambda'_0}^{-n-p}M_{\tau_0}(\ovx)\rp^\frac{1}{p}\delta_\bnd^{\frac{1}{2}\beta_0}
        \sum_{j=k_0}^\infty\eta_0^{j\frac{\beta_0}{2}}\\
    &=
    c\lp\vep_{\lambda'_0}^{-n-p}M_{\tau_0}(\ovx)\rp^\frac{1}{p}
        \lp\delta_\bnd\eta_0^{k_0}\rp^{\frac{1}{2}\beta_0}
        \lp1-\eta_0^{\frac{1}{2}\beta_0}\rp^{-1}.
\end{align*}
As the upper bound is independent of $k'$, we conclude that $\{g(\vx_j)\}_{j=0}^\infty$ is a Cauchy sequence and must converge to some value in $\re$, which we identify as $g(\ovx)\in\re$.

We now prove~\eqref{E:gHoldCont} for $0<\lambda<1$. First, suppose that $0<\lambda\le\lambda_0$, and let $0<\beta<\beta_0$ be given. Then $\tau_0-\beta p>\omega_0$, so we may select $k_\beta\in\whls$ such that $\alpha_{\delta_j}(\ovx)+(\tau_0-p\beta)\ge 0$, for all $j\ge k_\beta$. Let $\vx\in Q_{\lambda',\rho_{k_\beta}}(\ovx)$ be given. There exists a unique $k\ge k_\beta$ such that $\rho_{k+1}\le\|\ovx-\vx\|_{\ren}<\rho_k$. Since $Q_{\lambda_0}(\ovx)\subseteq Q_\lambda(\ovx)$, we may define $\{\vx'_j\}_{j=k}^\infty\subseteq Q_{\lambda}(\ovx)$ by $\vx'_k=\vx$ and $\vx'_j:=\vx_j$, for $j>k$. Put $\lambda':=\eta_0^2\lambda$. As argued above, \cref{L:gx'gxDiffBnd} implies
\begin{align*}
    |g(\ovx)-g(\vx)|
    &=
    \lim_{k'\to\infty}|g(\vx_{k'})-g(\vx_k)|
    \le
    c\vep_{\lambda'}^{-\frac{n+p}{p}}\sum_{j=k}^\infty
    \lp\delta_j^{\alpha_{\delta_j}(\ovx)}
        \nu^{s(\cdot),p}\lp\bll_{\delta_j}(\ovx)\rp\rp^\frac{1}{p}\\
    &\le
    c\lp\vep_{\lambda'}^{-n-p}M_{\tau_0}(\ovx)\rp^\frac{1}{p}
        \sum_{j=k}^\infty\lp\delta_j^{(\alpha_{\delta_j}(\ovx)-\beta p)+\beta p}\rp^\frac{1}{p}\\
    &\le
    c\lp\vep_{\lambda'}^{-n-p}M_{\tau_0}(\ovx)\rp^\frac{1}{p}\delta_\bnd^{\beta}
        \sum_{j=k}^\infty\eta_0^{j\beta}\\
    &\le
    c(3C_\bnd)^\beta\lp\vep_{\lambda'}^{-n-p}M_{\tau_0}(\ovx)\rp^\frac{1}{p}\lp1-\eta_0^\beta\rp^{-1}
        \lp\eta_0^k\rho_\bnd\rp^\beta\\
    &\le
    c\eta_0^{-\beta}\lp\vep_{\lambda'}^{-n-p}M_{\tau_0}(\ovx)\rp^\frac{1}{p}
        \lp1-\eta_0^\beta\rp^{-1}
        \|\ovx-\vx\|_{\ren}^\beta.
\end{align*}
This proves~\eqref{E:gHoldCont}, with $\delta_\beta:=\eta_0^{k_\beta}\rho_\bnd$. If on the other hand $\lambda_0<\lambda<1$, then $\vx\in Q_{\lambda_0}(\ovx)$. The same argument, with the sequence $\{\vx_j\}_{j=k}^\infty\subseteq Q_{\lambda_0}(\ovx)$ identified above, may be used.
\end{proof}

\begin{corollary}\label{C:LebPropTrace}
    Assume (H1) and (H2). Let $\tau_0\ge0$ and $\omega_0<\tau_0$ be given. If $\ovx_0\in A_{\omega_0,0}\cap\bnd_{\tau_0}$, then
\[
    \lim_{\rho\to0^+}\fslint{Q_{\lambda,\rho}(\ovx_0)}|u(\vx)-g(\ovx_0)|^p\dd\vx=0,
    \frl0<\lambda<\lambda_0.
\]
\end{corollary}
\begin{proof}
We use the notation from \cref{T:gGenResults}, and put $\theta_0:=\theta_\bnd(\ovx_0)$. Let $0<\beta<(\tau_0-\omega_0)/p$ and $0<\rho<\delta_\beta(\ovx_0)$ be given. Put
\[
    R_0:=\sup_{\vx\in Q_{\lambda_0,\rho}(\ovx_0)}d_{\partial\dom}(\vx)
    \nd
    R:=\sup_{\vx\in Q_{\lambda,\rho}(\ovx_0)}d_{\partial\dom}(\vx).
\]
In the argument that follows, the constant $c<\infty$ is independent of $\rho$. \cref{L:Hypotheses}(\ref{L:CorkscrewVol}) implies $|Q_{\lambda,\rho}|\ge c(1-\lambda/\lambda_0)^{n\theta_0}R_0^n$ and $R_0\ge(\eta_0\lambda_0\rho)^{\theta_0}$. For each $\vx\in Q_{\lambda,\rho}\bs Q_{\lambda_0,\rho}(\ovx_0)$, we have
\[
    (\lambda\|\ovx_0-\vx\|_{\ren})^{\theta_0}
    <
    d_{\partial\dom}(\vx)
    \le
    (\lambda_0\|\ovx_0-\vx\|_{\ren})^{\theta_0}
    \le(\lambda_0\rho)^{\theta_0}
    \le\eta_0^{-\theta_0}R_0.
\]
We deduce that $R_0\le R\le\eta_0^{-\theta_0}R_0$. Thus,
\begin{align*}
    &\fslint{Q_{\lambda,\rho}(\ovx_0)}|u(\vx)-g(\ovx_0)|^p\dd\vx\\
    &\qquad\le
    c\lp\fslint{Q_{\lambda,\rho}(\ovx_0)}|u(\vx)-g(\ovx)|^p\dd\vx
    +\fslint{Q_{\lambda,\rho}(\ovx_0)}|g(\vx)-g(\ovx_0)|^p\dd\vx\rp\\
    &\qquad\le
    c(1-\lambda/\lambda_0)^{-n\theta_0}\lp\frac{R^{p\ul{s}_\rho(\ovx_0)}}{R_0^n}\rp
        \nu^{s(\cdot),p}(Q_{\lambda,\rho}(\ovx_0))\\
    &\qqqquad\qqquad+
    cCM_{\tau_0}(\ovx_0)\fslint{Q_{\lambda,\rho}(\ovx_0)}\|\ovx_0-\vx\|^{p\beta}\dd\vx\\
    &\qquad\le
    c\lp R^{p\ul{s}_\rho(\ovx_0)-n}\nu(\bll_\rho(\ovx_0))
    +CM_{\tau_0}(\ovx_0)\rho^{p\beta}\rp.
\end{align*}
If $p\ul{s}_0(\ovx_0)>n$, then for $\rho>0$ sufficiently small, we find $p\ul{s}_\rho(\ovx_0)-n\ge (p\ul{s}_0(\ovx_0)-n)/2>0$ and
\[
    \fslint{Q_{\lambda,\rho}(\ovx_0)}|u(\vx)-g(\ovx_0)|^p\dd\vx
    \le
    c\lp\rho^{\frac{p\ul{s}_0(\ovx_0)-n}{2}}+C\rho^{p\beta}\rp
    M_{\tau_0}(\ovx_0).
\]
If, on the other hand, we have $p\ul{s}_0(\ovx_0)\le n\le n+p-1$, then for any $0<\rho<\delta_\beta$, we must have $p\ul{s}_\rho(\ovx_0)-n\le0$ and $R^{p\ul{s}_\rho(\ovx_0)-n}\le c\rho^{p\ul{s}_\rho(\ovx_0)-n}\le c\rho^{(p\ul{s}_\rho(\ovx_0)-n)\theta_0}$. Hence,
\begin{align*}
    \fslint{Q_{\lambda,\rho}(\ovx_0)}|u(\vx)-g(\ovx_0)|^p\dd\vx
    \le&
    c\lp\rho^{(p\ul{s}_\rho(\ovx_0)-n)\theta_0}\nu(\bll_\rho(\ovx_0))
    +CM_{\tau_0}(\ovx_0)\rho^{p\beta}\rp\\
    \le&
    c\lp\rho^{\ul{\alpha}_\rho(\ovx_0)}\rho^{p(\theta_0-1)}\nu(\bll_\rho(\ovx_0))
    +CM_{\tau_0}(\ovx_0)\rho^{p\beta}\rp\\
    \le&
    c\lp\rho^{(\tau_0-\omega_0)+p(\theta_0-1)}+C\rho^{p\beta}M_{\tau_0}(\ovx_0)\rp.
\end{align*}
In either case, the result follows.
\end{proof}

A straightforward application of \cref{T:gGenResults} provides a proof for the existence of a trace of $u$ on $\bnd$.
\begin{proof}[Proof for \cref{T:TraceExists}]
In addition to (H1) and (H2), we assume $\ul{\alpha}_0(\ovx)+t>0$, for $\haus^t$-a.e. $\ovx\in\bnd$. For each $\omega\in\re$, define $A_{\omega^-,0}:=\bigcup_{\omega'<\omega}A_{\omega',0}$. Set
\[
    \bnd':=\bigcup_{\tau\ge0}A_{\tau^-,0}\cap\bnd_\tau.
\]
As the sets $\{A_{\omega'}\}_{\omega'\in\re}$ are nested and Borel-measurable, we deduce that $\bnd'$ is also measurable. Moreover, we find $A_{t^-,0}\cap\bnd_t\subseteq\bnd'$. Thus, assumption (H3$'$) and~\eqref{E:limsupNuBound} implies $\haus^t(\bnd/(A_{t^-,0}\cap\bnd_t))=0$. Therefore $\haus^t(\bnd\bs\bnd')=0$. Define $Tu\in L(\bnd)$ by
\begin{equation}\label{E:TraceDef}
    Tu(\ov{\vx}):=\lc\begin{array}{ll}
        g(\ov{\vx}), & \ov{\vx}\in\bnd',\\
        0, & \vx\in\bnd\bs\bnd'.
    \end{array}\rt
\end{equation}
The measurability of $Tu$ is a consequence of the measurability of $\bnd'$ and $g$, in $\dom$. Since $\ovx\in A_{\tau^-,0}\cap\bnd_\tau$ implies there must be $\omega<\tau$ such that $\ovx\in A_{\omega,0}\cap\bnd_\tau$, \cref{T:TraceExists} follows from \cref{T:gGenResults}.
\end{proof}

\begin{definition}\label{D:uTrace}
Under assumptions (H1), (H2), and (H3$\,'$), we identify the function $Tu\in L(\bnd)$ defined in~\eqref{E:TraceDef} as the \emph{trace} of $u$ on $\bnd$.
\end{definition}

\section{Properties of the Trace}\label{S:Props}

We next focus on \cref{T:TraceProp}. Throughout this section, we use $Tu$ as provided by~\eqref{E:TraceDef}. Put $\eta_1:=\min\{\frac{1}{2},\eta_0\}$. For each $\omega\in\re$, $0<\lambda\le\lambda_0$, $0<\rho\le\rho_\bnd$ and $0<\delta\le\delta_\bnd$, set
\[
    A_{\omega,\delta}
    :=
    \lc\ovx\in\bnd: \ul{\alpha}_{\delta}(\ovx)+\omega\ge0\rc\in\class{B}(\partial\dom),
\]
and define $F_{\lambda,\rho}:\bnd\twoheadrightarrow\dom$ and $g_\rho,G_{\delta}:\bnd\to\re$ by
\[
    F_{\lambda,\rho}(\ovx):=Q_{\lambda,\rho}\bs\bll_{\eta_1\rho}(\ovx),
    \:\;
    g_\rho(\ovx):=\fslint{F_{\lambda_0,\rho}(\ovx)}g(\vx)\dd\vx,
    \text{ and }
    G_{\delta}(\ovx):=
        \nu^{s(\cdot),p}\lp\bll_{\delta}(\ovx)\rp.
\]
We note that, since $\vx\mapsto \|\vx\|_{\ren},d_{\partial\dom}(\vx)$ is a continuous function, assumption (H1) implies $F_{\lambda,\rho}(\ovx)$ has nonempty interior and $|F_{\lambda,\rho}(\ovx)|>0$, for all $\ovx\in\bnd$. We also see that $\ovx\in A_{\omega,\delta}\subseteq A_{\omega',\delta'}$, for all $\omega'\ge\omega$ and $0<\delta'\le\delta$.


To establish the regularity and Lebesgue property of the trace, we need a refinement of~\eqref{E:gHoldCont}.
\begin{lemma}\label{L:gxBarResult}
Assume (H1) and (H2). Let $\tau\ge0$, $\omega<\tau$, and $0<\lambda\le\lambda_0$ be given. With $k\in\whls$ and $0<\delta\le\eta_1^k\delta_\bnd$, put $\rho:=\frac{\delta}{3C_\bnd}$, and suppose that $\ovx\in A_{\omega,\delta}\cap\bnd_\tau$.  Then
\begin{equation}\label{E:gxgxBarDiffBnd}
    |Tu(\ovx)-g(\vx)|
    \le
    c\vep_{\lambda'}^{-\frac{n+p}{p}}\sum_{j=k}^\infty\lp\eta_1^{-j\omega}
        G_{\eta_1^j\delta_\bnd}(\ovx)\rp^\frac{1}{p},
    \frl\vx\in Q_{\lambda,\rho}(\ovx).
\end{equation}
Here $\lambda':=\eta_1^2\lambda$.
\end{lemma}
\begin{proof}
For $j\in\whls$, put $\rho_j:=\eta_1^j\rho_\bnd$ and $\delta_j:=\eta_1^j\delta_\bnd=3C_\bnd\rho_j$. There exists unique $k_1\ge k_0\ge k$ such that $\vx\in Q_{\lambda,\rho_{k_1}}\bs\bll_{\rho_{k_1+1}}(\ovx)$ and $\delta_{k_0+1}<\delta\le\delta_{k_0}$. We may select $\{\vx_j\}_{j=k_0}^\infty\subseteq Q_{\lambda,\rho_k}(\ovx)$ such that $\vx_j\in Q_{\lambda,\rho_j}\bs\bll_{\rho_{j+1}}(\ovx)$, for each $j\ge k_0$, and $\vx_{k_1}=\vx$. We see that $A_{\omega,\delta_j}\subseteq A_{\omega,\delta}\subseteq A_{\omega,0}$, for all $j\ge k_0+1$. Since $\omega<\tau$, \cref{T:gGenResults} implies $g(\vx_j)\to Tu(\ovx)$ in $Q_\lambda(\ovx)$, as $j\to\infty$. If $k_1=k_0$, we might have $\delta_{k_1}=\delta_{k_0}\ge\delta$ and $\ul{\alpha}_\delta(\ovx)\ge-\omega>\ul{\alpha}_{\delta_{k_0}}(\ovx)=\ul{\alpha}_{\delta_{k_1}}(\ovx)$. In which case, we find
\[
    \vx_{k_1+1},\vx_{k_1}
    \in
    Q_{\lambda,\rho}\bs\bll_{\eta_1\rho_{k_1+1}}(\ovx)
    \subseteq
    Q_{\lambda,\rho}\bs\bll_{\eta_1^2\rho}(\ovx)
\]
and may use \cref{L:guMeansDiff} to get
\begin{align*}
    |g(\vx_{k_1+1})-g(\vx)|
    =&
    |g(\vx_{k_1+1})-g(\vx_{k_1})|
    \le
    c\vep_{\lambda'}^{-\frac{n+p}{p}}\lp\delta^{\ul{\alpha}_{\delta}(\ovx)}
        \nu^{s(\cdot),p}\lp\bll_{\delta}(\ovx)\rp\rp^\frac{1}{p}\\
    \le&
    c\vep_{\lambda'}^{-\frac{n+p}{p}}\lp\delta_{k_1}^{-\omega}
        \nu^{s(\cdot),p}\lp\bll_{\delta_{k_1}}(\ovx)\rp\rp^\frac{1}{p}
\end{align*}
In any case, we have $\ul{\alpha}_{\delta_j}\ge-\omega$, for $j\ge k_0+1$. Using \cref{L:gx'gxDiffBnd} and passing to the limit as $k'\to+\infty$, we obtain
\begin{align*}
    |Tu(\ovx)-g(\vx)|
    \le&
    \lim_{k'\to\infty}|g(\vx_{k'})-g(\vx_{k_1)}|\\
    \le&
    c\vep_{\lambda'}^{-\frac{n+p}{p}}\sum_{j=k_1+1}^\infty\lp\delta_j^{\ul{\alpha}_{\delta_j}(\ovx)}
        \nu^{s(\cdot),p}\lp\bll_{\delta_j}(\ovx)\rp\rp^\frac{1}{p}
    +c\vep_{\lambda'}^{-\frac{n+p}{p}}\lp\delta_{k_1}^{\ul{\alpha}_{\delta}(\ovx)}
        \nu^{s(\cdot),p}\lp\bll_{\delta}(\ovx)\rp\rp^\frac{1}{p}\\
    \le&
    c\vep_{\lambda'}^{-\frac{n+p}{p}}\sum_{j=k_1}^\infty\lp\delta_j^{-\omega}
        \nu^{s(\cdot),p}\lp\bll_{\delta_j}(\ovx)\rp\rp^\frac{1}{p}.
\end{align*}
The result follows from $k_1\ge k_0\ge k$ and the definition of $G_{\delta_j}$.
\end{proof}
We will also need
\begin{lemma}\label{L:TugDiffBndG}
Assume (H1) and (H2). Let $\tau\ge0$, $\omega<\tau$, and $0<\lambda\le\lambda_0$ be given. With $k\in\whls$ and $0<\delta\le\eta_1^k\delta_\bnd$, put $\rho:=\frac{\delta}{3C_\bnd}$, and suppose that $\bnd''\subseteq A_{\omega,\delta}\cap\bnd_\tau$ is $\haus^t$-measurable and that $F:\bnd''\twoheadrightarrow\dom$ satisfies $F(\ovx)\subseteq Q_{\lambda,\rho}(\ovx)$ and $|F(\ovx)|>0$, for each $\ovx\in\bnd''$. Then, with $\lambda':=\eta_1^2\lambda$, we have
\[
    \slint{\bnd''}
        \fslint{F(\ovx)}\ns|Tu(\ovx)-g(\vx)|^p\dd\vx\dd\haus^t(\ovx)
    \le
    c\vep_{\lambda'}^{-n-p}\ns
        \lp\sum_{j=k}^\infty\lp\eta_1^{-j\omega}\ns
            \slint{\bnd''}\ns
            G_{\eta_1^j\delta_\bnd}(\ovx)\dd\haus^t(\ovx)\rp^\frac{1}{p}\rp^p\ns.
\]
\end{lemma}
\begin{proof}
We may assume that the integral on in the lower bound is positive. Define $\{\delta_j\}_{j=1}^\infty\subset(0,\delta_\bnd]$ and $\{\rho_j\}_{j=1}^\infty\subset(0,\rho_\bnd]$ as in \cref{L:gxBarResult}. For each $\ovx\in\bnd_\tau$ and $\vx\in Q_{\lambda,\rho}(\ovx)$, \cref{L:gxBarResult} provides
\[
    |Tu(\ovx)-g(\vx)|
    \le
    c\vep_{\lambda'}^{-\frac{n+p}{p}}
        \sum_{j=k}^\infty\delta_j^{-\frac{\omega}{p}}\lp G_{\delta_j}(\ovx)\rp^\frac{1}{p}.
\]
If $p=1$, we are done after taking the mean of both sides over $F(\ovx)$ and integrating over $\bnd''$. Otherwise, the monotone convergence theorem and H\'{o}lder's inequality yields
\begin{align*}
    &\slint{\bnd''}\fslint{F(\ovx)}|Tu(\ovx)-g(\vx)|^p\dd\vx\dd\haus^t(\ovx)\\
    &\qquad\le
    c\vep_{\lambda'}^{-\frac{n+p}{p}}\sum_{j=k}^\infty\delta_j^{-\frac{\omega}{p}}
    \slint{\bnd''}\lp\fslint{F(\ovx)}|Tu(\ovx)-g(\vx)|^{p-1}\dd\vx\rp
        \lp G_{\delta_j}(\ovx)\rp^\frac{1}{p}\dd\haus^t(\ovx)\\
    &\qquad\le
    c\vep_{\lambda'}^{-\frac{n+p}{p}}\sum_{j=k}^\infty\delta_j^{-\frac{\omega}{p}}
    \lp\slint{\bnd''}\lp\fslint{F(\ovx)}|Tu(\ovx)-g(\vx)|^{p-1}\dd\vx
        \rp^\frac{p}{p-1}\dd\haus^t(\ovx)\rp^\frac{p-1}{p}\\
    &\qqqquad\qqqquad\qqquad\qquad
    \times\lp\slint{\bnd''}G_{\delta_j}(\ovx)\dd\haus^t(\ovx)\rp^\frac{1}{p}.
\end{align*}
We may apply Jensen's inequality to the first integral and obtain
\begin{multline*}
     \slint{\bnd''}\flint{F(\ovx)}\ns|Tu(\ovx)-g(\vx)|^p\dd\vx\dd\haus^t(\ovx)\\
     \le
     c\vep_{\lambda'}^{-\frac{n+p}{p}}\ns
     \lp\slint{\bnd''}\fslint{F(\ovx)}\ns\!|Tu(\ovx)-g(\vx)|^p\dd\vx
        \dd\haus^t(\ovx)\!\rp^{\!\frac{p-1}{p}}\ns
     \sum_{j=k}^\infty\delta_j^{-\frac{\omega}{p}}\ns
     \lp\slint{\bnd''}\ns\! G_{\delta_j}(\ovx)\dd\haus^t(\ovx)\!\rp^{\!\frac{1}{p}}\nss,
\end{multline*}
which implies the result after dividing both sides by the term in parentheses.

\end{proof}
Finally, we need the followinng lemma whose proof is the same as an analogous result in~\cite{MayGuyFen:18a}.
\begin{lemma}\label{L:supGBnd}
Assume (H4${}^\prime$). For each $0<\delta\le\delta_\bnd$ and $\haus^t$-measurable $\bnd''\subseteq\bnd$, we have
\[
    \slint{\bnd''}G_{\delta}(\ovx)\dd\haus^t(\ovx)
    \le
    c\delta^t\nu^{s(\cdot),p}\lp\bll_{2\delta}(\bnd'')\rp.
\]
Here $c$ is independent of $\delta_0.$
\end{lemma}
\begin{proof}
\cref{L:BoundedOverlap} delivers a countable index set $I$ and $\{\ov{\vy}_i\}_{i\in I}\subseteq\bnd''$ such that $\bnd''\subseteq\bigcup_{i\in I}\bll_\delta(\ov{\vy}_i)$ and $\sup_{\vx\in\ren}\sum_{i\in I}\chi_{\bll_{2\delta}(\ov{\vy}_i)}(\vx)\le c$. Now, for each $i\in I$ and $\ovx\in\bnd'\cap\bll_{\delta}(\ov{\vy}_i)$, we find $\dom\cap\bll_{\delta}(\ov{\vx})\subseteq\dom\cap\bll_{2\delta}(\ov{\vy}_i)$. Hence,
\[
    \sup_{\ovx\in\bnd'\cap\bll_{\delta}(\ov{\vy}_i)}G_{\delta}(\ov{\vx})
    =
    \sup_{\ovx\in\bnd'\cap\bll_{\delta}(\ov{\vy}_i)}\nu^{s(\cdot),p}(\bll_{\delta}(\ovx))
    \le
    \nu^{s(\cdot),p}\lp\bll_{2\delta}(\ov{\vy}_i)\rp.
\]
Using the upper Ahlfor's regularity assumption for $\bnd$, and thus for $\bnd''$, and the bounded overlap property of the family $\lc\bll_{2\delta}(\ov{\vy}_i)\rc_{i\in I}$, we obtain
\begin{align*}
    \slint{\bnd''}G_{\delta}(\ovx)\dd\haus^t(\ovx)
    &\le
    \sum_{i\in I}\slint{\bnd''\cap\bll_{\delta}(\ov{\vy}_i)}G_{\delta}(\ovx)\dd\haus^t(\ovx)
    \le
    \sum_{i\in I}\haus^t\lp\bnd''\cap\bll_{\delta}(\ov{\vy}_i)\rp
        \nu^{s(\cdot),p}\lp\bll_{2\delta}(\ov{\vy}_i)\rp\\
    &\le
    A_\bnd\delta^{t}
        \sum_{i\in I}\nu^{s(\cdot),p}\lp\bll_{2\delta}(\ov{\vy}_i)\rp
    \le
    c\delta^{t}
        \nu^{s(\cdot),p}\lp\bll_{2\delta}(\bnd'')\rp.
\end{align*}
\end{proof}
\begin{remark}\label{R:supGBnd}
In the last line of the proof, we see that, in the upper bound, the term $\nu^{s(\cdot),p}\lp\bll_{2\delta}(\bnd'')\rp$ can be made more precise with $\nu^{s(\cdot),p}\lp U\rp$, with $U:=\bigcup_{i\in I}\bll_{2\delta}(\ovy_i)$.
\end{remark}

\begin{lemma}\label{L:TugLpBnd}
Assume (H1), (H2), and (H4$'$). Let $\omega<t$, and $0<\lambda\le\lambda_0$ be given. With $k\in\whls$ and $0<\delta\le\eta_1^k\delta_\bnd$, suppose that $\bnd''\subseteq A_{\omega,\delta}\cap\bnd_t$ is $\haus^t$-measurable. Then, for each $0<\rho\le\frac{\delta}{3C_\bnd}$, we have
\[
    \|Tu-g_\rho\|_{L^p(\bnd'')}
    \le
    C_1\eta_1^{k\lp\frac{t-\omega}{p}\rp}|u|_{\class{N}^{s(\cdot),p}\lp\bll_{2\eta_1^k\delta_0}(\bnd'')\rp},
\]
with
\[
    C_1=C_1(\lambda,\eta_1,\omega)=c\vep_{\lambda'}^{-\frac{n+p}{p}}
        \lp1-\eta_1^{\frac{t-\omega}{p}}\rp^{-1}
    \nd
    \lambda':=\eta_1^2\lambda.
\]
\end{lemma}
\begin{proof}
Again, for each $j\in\nats$, put $\wh{\rho}_j:=\eta_1^j\delta_\bnd$ and $\rho_j:=\eta_1^j\rho_\bnd$. If $1<p<\infty$, then we apply Jensen's inequality and then H\"{o}lder's inequality to obtain
\begin{align*}
    \|Tu-g_{\rho}\|^p_{L^p(\bnd'')}
    &\le
    \slint{\bnd''}\lp|Tu(\ov{\vx})-g_{\rho}(\ov{\vx})|^{p-1}
        \fslint{F_{\lambda,\rho}(\ov{\vx})}|Tu(\ov{\vx})-g(\vx)|\dd\vx\rp\dd\haus^t(\ovx)\\
    &\le
    \|Tu-g_{\rho}\|^{p-1}_{L^p(\bnd'')}
        \lp\slint{\bnd''}
        \lp\fslint{F_{\lambda,\rho}(\ov{\vx})}|Tu(\ov{\vx})-g(\vx)|\dd\vx\rp^p\dd\haus^t(\ovx)\rp^\frac{1}{p}\\
    &\le
    \|Tu-g_{\rho}\|^{p-1}_{L^p(\bnd'')}
        \lp\slint{\bnd''}
        \fslint{F_{\lambda,\rho}(\ov{\vx})}|Tu(\ov{\vx})-g(\vx)|^p
            \dd\vx\dd\haus^t(\ovx)\rp^\frac{1}{p}.
\end{align*}
If $p=1$, then the above inequality follows from Jensen's inequality alone. \cref{L:TugDiffBndG}, with $F=F_{\lambda,\rho}$, allows us to continue with
\begin{align}
\nonumber
    \|Tu-g_{\rho}\|_{L^p(\bnd'')}
    &\le
    \slint{\bnd''}
        \fslint{F_{\lambda,\rho}(\ov{\vx})}|Tu(\ov{\vx})-g(\vx)|^p\dd\vx\dd\haus^t(\ovx)\\
\label{E:TugrkLpBnd1}
    &\le
    c\vep_{\lambda'}^{-\frac{n+p}{p}}
        \sum_{j=k}^\infty\lp\delta_j^{-\omega}
            \slint{\bnd''}G_{\delta_j}(\ovx)\dd\haus^t(\ovx)\rp^\frac{1}{p}.
\end{align}

Next, \cref{L:supGBnd} provides
\begin{align*}
    \|Tu-g_{\rho}\|_{L^p(\bnd'')}
    &\le
    c\vep_{\lambda'}^{-\frac{n+p}{p}}
    \sum_{j=k}^\infty\delta_j^\frac{t-\omega}{p}
        \lp\nu^{s(\cdot),p}\lp\bll_{2\delta_j}(\bnd'')\rp\rp^\frac{1}{p}\\
    &\le
    c\vep_{\lambda'}^{-\frac{n+p}{p}}\delta_k^\frac{t-\omega}{p}
        \lp\nu^{s(\cdot),p}\lp\bll_{2\delta_k}(\bnd'')\rp\rp^\frac{1}{p}
        \sum_{j=0}^\infty\eta_1^{j\lp\frac{t-\omega}{p}\rp}.
\end{align*}
Since $\omega<t$, the series above is convergent and the result follows.
\end{proof}

\cref{T:TraceProp}(a) is a consequence of our first main result for this section.
\begin{theorem}\label{T:uInWbp}
Assume (H1), (H2), and (H4$'$). Let $\omega<t$, and $0<\delta\le\delta_\bnd$ be given. Suppose that $\bnd''\subseteq A_{\omega,\delta}\cap\bnd_t$ is $\haus^t$-measurable. Then $u\in W^{\beta,p}(\bnd'')$, for each $0<\beta<(t-\omega)/p$.
\end{theorem}
\begin{remark}
The bound for $\|Tu\|_{L^p(\bnd'')}$ is provided in~\eqref{E:TuLpBnd}, and the bound for $|Tu|_{W^{\beta,p}(\bnd'')}$ is the sum of the two bounds in~\eqref{E:WapTraceBnd2} and~\eqref{E:WapTraceBnd3}.
\end{remark}
\begin{proof}
We need to verify
\begin{equation}\label{E:TuGoal}
    \slint{\bnd''}\ns|Tu(\ovx)|^p\dd\haus^t(\ovx)<\infty
    \nd
    \slint{\bnd''}\slint{\bnd''}\ns
    \frac{|Tu(\ov{\vy})-Tu(\ov{\vx})|^p}{\|\ov{\vy}-\ov{\vx}\|^{t+p\beta}}
    \dd\haus^t(\ov{\vy})\dd\haus^t(\ov{\vx})<\infty.
\end{equation}
 Our arguments are similar to those used in~\cite{MayGuyFen:18a}. Recall that $\eta_1:=\min\{\frac{1}{2},\eta_0\}$. Put $\lambda_j:=\eta_1^{2j}\lambda_0$, for $j\in\whls$, and $\rho:=\frac{\delta}{3C_\bnd}$. Define $\{\delta_j\}_{j=1}^\infty\subset(0,\delta_\bnd]$ and $\{\rho_j\}_{j=1}^\infty\subset(0,\rho_\bnd]$ as in the proofs for above lemmas.

First, we argue that $Tu\in L^p(\bnd'')$. Let $k\in\whls$ be given. By \cref{L:TugLpBnd}, with $\eta=\eta_1$ and $\lambda=\lambda_0$, we have
\[
    \|Tu\|_{L^p(\bnd'')}
    \le
    \|g_{\rho_k}\|_{L^p(\bnd'')}
        +\|Tu-g_{\rho_k}\|_{L^p(\bnd'')}
    \le
    \|g_{\rho_k}\|_{L^p(\bnd'')}+C_1\delta_{k}^\frac{t-\omega}{p}
        |u|_{\class{N}^{s(\cdot),p}\lp\bll_{2\delta_{k}}(\bnd'')\rp}.
\]
We need to produce a bound for $\|g_{\rho_k}\|_{L^p(\bnd)}$.

For each $k\in\whls$, we see that
\[
    \vx\in F_{\lambda_0,\rho_k}(\ov{\vx})
    \Longrightarrow
    d_{\partial\dom}(\vx)\ge\lp\lambda_0\|\vx-\ov{\vx}\|_{\ren}\rp^{\theta_\bnd(\ovx)}
        >(\lambda_0\eta_0\rho_k)^{\theta_\bnd(\ovx)}.
\]
Hence $|\Psi(\vx)|\ge(\lambda_1\rho_k)^{n\theta_\bnd(\ovx)}$.
We also observe that
\begin{align*}
    \ovx\in\bnd''\subseteq\bnd,\:\;\vx\in F_{\lambda_0,\rho_k}(\ovx),\text{ and }\vy\in\Phi(\ovx)
    &\Longrightarrow
    \vx\in\bll_{\rho_k}(\bnd'')\text{ and }\vy\in\bll_{\frac{1}{6}d_{\partial\dom}(\vx)}(\vx)\\
    &\Longrightarrow
    \vy\in\bll_{2\rho_k}(\bnd'')\subseteq\bll_{2\delta_k}(\bnd'').
\end{align*}
Put $\ov{\theta}_{\bnd''}:=\sup_{\ovx\in\bnd''}|\theta_\bnd(\ovx)|<\infty$. Using Jensen's inequality, we conclude that
\begin{align*}
    \slint{\bnd''}|g_{\rho_k}(\ovx)|^p\dd\haus^t(\ovx)
    &\le
    \slint{\bnd''}\fslint{F_{\lambda_0,\rho_k}(\ovx)}
        \lp\fslint{\Phi(\vx)}|u(\vy)|\dd\vy\rp^p\dd\vx\dd\haus^t(\ovx)\\
    &\le
    (\lambda_1\rho_k)^{-np\ov{\theta}_{\bnd''}}\haus^t(\bnd'')
        \|u\|^p_{L^1\lp\bll_{2\delta_k}(\bnd'')\rp}<\infty.
\end{align*}
Thus
\begin{equation}\label{E:TuLpBnd}
    \|Tu\|_{L^p(\bnd'')}
    \le
    (\lambda_1\rho_k)^{-n\ov{\theta}_{\bnd''}}\haus^t(\bnd'')^\frac{1}{p}
        \|u\|_{L^1\lp\bll_{2\delta_k}(\bnd'')\rp}
    +C_1\delta_k^\frac{t-\omega}{p}|u|_{\class{N}^{s(\cdot),p}\lp\bll_{2\delta_k}(\bnd'')\rp}<\infty.
\end{equation}

We turn to establishing the second part of~\eqref{E:TuGoal}. For each $\rho>0$, define $\bnd''_\rho:\bnd''\twoheadrightarrow\bnd$ by
\[
    \bnd''_\rho(\ovx):=\bnd''\cap\bll_\rho(\ovx)\bs\bll_{\eta_0\rho}(\ovx).
\]
We also introduce $\bnd''_{1,\rho},\bnd''_{2,\rho}:\bnd\twoheadrightarrow\bnd$ given by
\[
    \bnd''_{1,\rho}(\ovx):=\lc\ovy\in\bnd''_\rho(\ovx):\theta_\bnd(\ovx)\le\theta_\bnd(\ovy)\rc
    \nd
    \bnd''_{2,\rho}(\ovx):=\bnd''_\rho\bs\bnd''_{1,\rho}(\ovx).
\]

With $k\in\nats$, let $\ovx\in\bnd''$, $\ovy\in\bnd''_{1,\rho_k}(\ovx)$, and $\vx\in F_{\lambda_0,\rho_{k+2}}(\ovx)$ be given. Clearly, $\eta_1\lp1+\eta_1^2\rp<1$, so
\[
    \|\ovy-\vx\|_{\ren}
    \le\|\ovy-\ovx\|_{\ren}+\|\ovx-\vx\|_{\ren}
    <
    \rho_k+\rho_{k+2}\le\eta_1^k\lp1+\eta_1^2\rp\rho<\rho_{k-1}.
\]
Also,
\[
    \|\ovy-\vx\|_{\ren}\ge\|\ovy-\ovx\|_{\ren}-\|\ovx-\vx\|_{\ren}
    >\rho_{k+1}-\rho_{k+2}=\lp1-\eta_1\rp\rho_{k+1}\ge\rho_{k+2}.
\]
Furthermore, since $\ovy\in\bnd''_{1,\rho_k}(\ovx)$, we have $\theta_\bnd(\ovy)\ge\theta_\bnd(\ovx)$, and thus
\[
    d_{\partial\dom}(\vx)
    \ge
    \lp\lambda_0\|\ovx-\vx\|_{\ren}\rp^{\theta_\bnd(\ovx)}
    \ge
    \lp\lambda_0\rho_{k+3}\rp^{\theta_\bnd(\ovy)}
    >\lp\lambda_0\eta_1^4\|\ovy-\vx\|_{\ren}\rp^{\theta_\bnd(\ovy)}
    =\lp\lambda_2\|\ovy-\vx\|_{\ren}\rp^{\theta_\bnd(\ovy)}.
\]
Thus $\vx\in Q^{\theta_\bnd(\ovy)}_{\lambda_2,\rho_{k-1}}\bs\bll_{\rho_{k+2}}(\ovy)$. As $k\in\nats$ and $\ovx\in F_{\lambda_0,\rho_{k+2}}(\ovx)$ were arbitrary, we conclude that
\begin{equation}\label{E:FsubE}
    F_{\lambda_0,\rho_{k+2}}(\ovx)
        \subseteq Q^{\theta_\bnd(\ovy)}_{\lambda_2,\rho_{k-1}}\bs\bll_{\rho_{k+2}}(\ovy),
    \frl
    k\in\nats\text{ and }\ovy\in\bnd''_{1,\rho_k}(\ovx).
\end{equation}

We may write
\begin{multline*}
    \slint{\bnd''}\slint{\bnd''}\frac{|Tu(\ov{\vy})-Tu(\ov{\vx})|^p}
        {\|\ov{\vy}-\ov{\vx}\|_{\ren}^{t+\beta p}}\dd\haus^t(\ov{\vy})\dd\haus^t(\ov{\vx})\\
    =
    \sum_{\ell=1}^2\sum_{k=-\infty}^\infty\slint{\bnd''}
        \slint{\bnd''_{\ell,\rho_k}(\ov{\vx})}\frac{|Tu(\ov{\vy})-Tu(\ov{\vx})|^p}
        {\|\ov{\vy}-\ov{\vx}\|_{\ren}^{t+\beta p}}\dd\haus^t(\ov{\vy})\dd\haus^t(\ov{\vx}).
\end{multline*}

We focus on the integrals with $k\in\nats$ first. We notice that if $\ovy\in\bnd''_{2,\rho}(\ovx)$, then $\theta_\bnd(\ovy)<\theta_\bnd(\ovx)$ and $\eta_1\rho\le\|\ovy-\ovx\|_{\ren}<\rho$, so $\ovx\in\bnd''_{1,\rho}(\ovy)$. This and the Fubini-Tonelli theorem implies
\begin{multline*}
    \slint{\bnd''}\slint{\bnd''_{2,\rho_k}(\ov{\vx})}\frac{|Tu(\ov{\vy})-Tu(\ov{\vx})|^p}
        {\|\ov{\vy}-\ov{\vx}\|_{\ren}^{t+\beta p}}\dd\haus^t(\ov{\vy})\dd\haus^t(\ov{\vx})\\
    \le
    \slint{\bnd''}\slint{\bnd''_{1,\rho_k}(\ov{\vx})}\frac{|Tu(\ov{\vy})-Tu(\ov{\vx})|^p}
        {\|\ov{\vy}-\ov{\vx}\|_{\ren}^{t+\beta p}}\dd\haus^t(\ov{\vy})\dd\haus^t(\ov{\vx}),
\end{multline*}
where we have relabeled $\ovx\leftrightarrow\ovy$ in the last integral. Thus
\begin{align}
\nonumber
    &\sum_{k=1}^\infty
        \slint{\bnd''}\slint{\bnd''_{\rho_k}(\ovx)}\nss\frac{|Tu(\ov{\vy})-Tu(\ov{\vx})|^p}
        {\|\ov{\vy}-\ov{\vx}\|_{\ren}^{t+\beta p}}\dd\haus^t(\ov{\vy})\dd\haus^t(\ov{\vx})\\
\nonumber
    &\qqquad=
    2\sum_{k=1}^\infty\slint{\bnd''}
        \slint{\bnd''_{1,\rho_k}(\ov{\vx})}\nss\frac{|Tu(\ov{\vy})-Tu(\ov{\vx})|^p}
        {\|\ov{\vy}-\ov{\vx}\|_{\ren}^{t+\beta p}}\dd\haus^t(\ov{\vy})\dd\haus^t(\ov{\vx})\\
\label{E:WapTraceBnd1}
    &\qqquad\le
    c\sum_{k=1}^\infty\underbrace{\slint{\bnd''}
        \slint{\bnd''_{1,\rho_k}(\ovx)}\nss\frac{|Tu(\ov{\vy})-g_{\rho_k}(\ovy)|^p}
        {\|\ov{\vy}-\ov{\vx}\|_{\ren}^{t+\beta p}}\dd\haus^t(\ovy)\dd\haus^t(\ovx)}_{=:I_{1,k}}\\
\nonumber
    &\qqqquad\qqquad+
    c\sum_{k=1}^\infty\underbrace{\slint{\bnd''}
        \slint{\bnd''_{1,\rho_k}(\ovx)}\nss\frac{|g_{\rho_{k+1}}(\ovx)-g_{\rho_k}(\ovy)|^p}
        {\|\ov{\vy}-\ov{\vx}\|_{\ren}^{t+\beta p}}\dd\haus^t(\ovy)\dd\haus^t(\ovx)}_{=:I_{2,k}}\\
\nonumber
    &\qqqquad\qqquad+
    c\sum_{k=1}^\infty\underbrace{\slint{\bnd''}
        \slint{\bnd''_{1,\rho_k}(\ovx)}\nss\frac{|Tu(\ov{\vx})-g_{\rho_{k+1}}(\ovx)|^p}
        {\|\ov{\vy}-\ov{\vx}\|_{\ren}^{t+\beta p}}\dd\haus^t(\ovy)\dd\haus^t(\ovx)}_{=:I_{3,k}}.
\end{align}

We first establish bounds for the  integrals $I_{1,k}$ and $I_{3,k}$, which can be handled in similar manners. For $I_{1,k}$, we have $\|\ovy-\ovx\|_{\ren}\ge\rho_{k+1}$. With the Fubini-Tonelli theorem and the upper Ahlfors-regularity, we obtain
\begin{align*}
    I_{1,k}
    \le&
    \rho_{k+1}^{-t-\beta p}\slint{\bnd''}\slint{\bnd_{2,k}(\ovy)}|Tu(\ov{\vy})-g_{\rho_k}(\ovy)|^p
        \dd\haus^t(\ovx)\dd\haus^t(\ovy)\\
    \le&
    A_\bnd\eta_1^{-t-\beta p}\rho_k^{-\beta p}\|Tu-g_{\rho_k}\|_{L^p(\bnd'')}^p\\
    \le&
    cC_1(\lambda_0,\eta_1,\omega)\eta_1^{-t-\beta p}\rho_k^{(t-\omega)-\beta p}
        |u|^p_{\class{N}^{s(\cdot),p}(\bll_{2\delta_k}(\bnd''))}.
\end{align*}
For the last line, we applied \cref{L:TugLpBnd}, with $\lambda=\lambda_0$ and $\delta=\delta_k$. For $I_{3,k}$, the same argument produces
\[
    I_{3,k}
    \le
    cC_1(\lambda_0,\eta_1,\omega)\eta_1^{-\omega-\beta p}\rho_k^{(t-\omega)-\beta p}
        |u|^p_{\class{N}^{s(\cdot),p}(\bll_{2\delta_{k+1}}(\bnd''))}.
\]

Turning to $I_{2,k}$, we again have $\|\ovy-\ovx\|_{\ren}\ge\rho_{k+1}$. This and Jensen's inequality produces the bound
\begin{equation}\label{E:I2kBnd1}
    I_{2,k}
    \le
    \rho_{k+1}^{-t-\beta p}\slint{\bnd''}\slint{\bnd_{1,\rho_k}(\ovx)}
    \fslint{F_{\lambda_0,\rho_k}(\ovy)}\fslint{F_{\lambda_0,\rho_{k+2}}(\ovx)}
        \ns|g(\vy)-g(\vx)|^p\dd\vx\dd\vy\dd\haus^t(\ovy)\dd\haus^t(\ovx).
\end{equation}
For each $\ovx\in\bnd''$, we find $\ovy\in\bnd_{1,\rho_k}(\ovx)$. Thus
\[
    F_{\lambda_0,\rho_{k+2}}(\ovx)
        \subseteq Q_{\lambda_2,\rho_{k-1}}\bs\bll_{\eta_1^3\rho_{k-1}}(\ovy).
\]
Since $\lambda_2\le\lambda_0$, from its definition, we also have
\[
    F_{\lambda_0,\rho_{k+1}}(\ovy)
    =Q_{\lambda_2,\eta_1^2\rho_{k-1}}\bs\bll_{\eta_1^3\rho_{k-1}}(\ovy)
        \subseteq Q_{\lambda_2,\rho_{k-1}}\bs\bll_{\eta_1^3\rho_{k-1}}(\ovy).
\]
Lemma~\ref{L:guMeansDiff}, with $\lambda=\lambda_2$, $\eta=\eta_1^3$, and $\rho=\rho_{k-1}$, implies
\[]
    |g(\vy)-g(\vx)|
    \le
    \vep^{-n-p}_{\lambda_3}\delta_{k-1}^{\ul{\alpha}_{\delta_{k-1}}}
        \nu^{s(\cdot),p}(\bll_{\delta_{k-1}}(\ovy))
    \le
    c\vep_{\lambda_3}^{-n-p}\eta_1^{\omega}\rho_k^{-\omega}G_{\delta_{k-1}}(\ovy).
\]
Returning to~\eqref{E:I2kBnd1}, we apply the Fubini-Tonelli theorem and \cref{L:supGBnd} and use the upper Ahlfors-regularity to produce
\begin{align*}
    I_{2,k}
    \le&
    c\vep_{\lambda_3}^{-n-p}\eta_1^{-t+\omega-\beta p}\rho_k^{-\omega-\beta p}
        \slint{\bnd''}G_{\delta_{k-1}}(\ovy)\dd\haus^t(\ovy)\\
    \le&
    c\vep_{\lambda_3}^{-n-p}\eta_1^{-2t+\omega-\beta p}\rho_k^{(t-\omega)-\beta p}
        |u|^p_{\class{N}^{s(\cdot),p}(\bll_{2\delta_{k-1}}(\bnd''))}.
\end{align*}

Using the bounds for $I_{1,k}$, $I_{2,k}$, and $I_{3,k}$ in~\eqref{E:WapTraceBnd1}, we conclude that
\begin{align}
\nonumber
    &\sum_{k=1}^\infty\slint{\bnd''}\slint{\bnd''_{\rho_k}(\ovx)}
        \frac{|Tu(\ovy)-Tu(\ovx)|^p}{\|\ovy-\ovx\|^{t+\beta p}}
        \dd\haus^t(\ovy)\dd\haus^t(\ovx)\\
\nonumber
    &\qqqquad\le
    cC_1\vep_{\lambda_3}^{-n-p}\eta_1^{-2t+\omega-\beta p}\sum_{k=1}^\infty\rho_k^{(t-\omega)-\beta p}
        |u|^p_{\class{N}^{s(\cdot),p}(\bll_{2\delta_{k-1}}(\bnd''))}\\
\label{E:WapTraceBnd2}
    &\qqqquad\le
    C'_2|u|^p_{\class{N}^{s(\cdot),p}(\bll_{2\delta_\bnd}(\bnd''))},
\end{align}
since $\beta p<t-\omega$. Here, we have put
\begin{align*}
    C'_2=C'_2(\lambda_0,\eta_1,\delta_\bnd,\omega,\beta)
    :=&
    cC_1(\lambda_0,\eta_1,\omega)\delta_\bnd^{(t-\omega)-\beta p}
        \eta_1^{-t-2\beta p}\lp1-\eta_1^{(t-\omega)-\beta p}\rp^{-1}\\
    \le&
    cC_1(\lambda_0,\eta_1,\omega)\eta_1^{-3t+2\omega}\lp1-\eta_1^{(t-\omega)-\beta p}\rp^{-1}.
\end{align*}

We now turn to establishing a bound for
\[
    \sum_{k=-\infty}^1
    \slint{\bnd''}\lint{\bnd''_{\rho_k}(\ovx)}\frac{|Tu(\ov{\vy})-Tu(\ov{\vx})|^p}
        {\|\ov{\vy}-\ov{\vx}\|_{\ren}^{t+\beta p}}\dd\haus^t(\ov{\vy})\dd\haus^t(\ov{\vx}).
\]
In this case, we have $\|\ovy-\ovx\|_{\ren}\ge\rho_{k+1}$, for each $\ovy\in\bnd_{\rho_k}(\ovx)$ and $k\le 1$. Consequently,
\begin{align}
\nonumber
    &\sum_{k=-\infty}^1\slint{\bnd''}
    \slint{\bnd''_{\rho_k}(\ovx)}\frac{|Tu(\ov{\vy})-Tu(\ov{\vx})|^p}
        {\|\ov{\vy}-\ov{\vx}\|_{\ren}^{t+\beta p}}\dd\haus^t(\ov{\vy})\dd\haus^t(\ov{\vx})\\
\nonumber
    &\qqqquad\le
    \sum_{k=-\infty}^1\rho_{k+1}^{-t-\beta p}
    \slint{\bnd''}\slint{\bnd''_{\rho_k}(\ovx)}|Tu(\ov{\vy})-Tu(\ov{\vx})|^p
        \dd\haus^t(\ov{\vy})\dd\haus^t(\ov{\vx})\\
\nonumber
    &\qqqquad\le
    c\sum_{k=-\infty}^1\rho_{k+1}^{-t-\beta p}
    \slint{\bnd''}\haus^t(\bnd''_{\rho_k}(\ovx))|Tu(\ovx)|^p\dd\haus^t(\ov{\vx})\\
\nonumber
    &\qqqquad\le
    cA_\bnd\eta_1^{-t-\beta p}\|Tu\|^p_{L^p(\bnd'')}
        \sum_{k=-\infty}^1\rho_k^{-\beta p}\\
\label{E:WapTraceBnd3}
    &\qqqquad
    =c\delta_\bnd^{-\beta p}\eta_1^{-t-2\beta p}
        \|Tu\|^p_{L^p(\bnd'')}\sum_{k=0}^\infty\eta_1^{k\beta p}
    =
    C''_2\|Tu\|^p_{L^p(\bnd)},
\end{align}
with
\[
    C''_2=C_2''(\lambda_0,\eta_1,\delta_\bnd,\beta)
    :=c\delta_\bnd^{-\beta p}\eta_1^{-t-2\beta p}\lp1-\eta_1^{\beta p}\rp^{-1}.
\]
Recalling~\eqref{E:TuLpBnd} and~\eqref{E:WapTraceBnd2}, the result is proved.
\end{proof}

This section's second main result will be used to establish \cref{T:TraceProp}(b).
\begin{theorem}\label{T:LebProp}
Assume (H1), (H2), (H4$'$), and (H4$''$). There exists $\delta',\rho'>0$ such that $\partial\dom\cap\bll_{\rho'}(\ovx_0)\subseteq A_{\omega,\delta'}\cap\bnd_t$, with $\omega<t-n(\theta_\bnd(\ovx_0)-1)$.
Then,
\[
    \lim_{\rho\to0^+}\fslint{\dom_\rho(\ovx_0)}|Tu(\ovx_0)-u(\vx)|^p\dd\vx
    =0.
\]
\end{theorem}
\begin{remark}\label{R:TraceLebProp}
\begin{enumerate}[(a)]
    \item The assumption $\partial\dom\cap\bll_{\rho'}(\ovx_0)\subseteq A_{\omega,\delta'}\cap\bnd_t$ requires $\bnd\cap\bll_\rho(\ovx_0)=\partial\dom\cap\bll_\rho(\ovx_0)$, for all $0<\rho\le\rho'$.
    \item\label{R:uIntegrability} The result implies there exists $\rho_0>0$ such that
\begin{align*}
    &\slint{\dom_{\rho_0}(\ovx_0)}|u(\vx)|^p\dd\vx
    \le
    c\slint{\dom_{\rho_0}(\ovx_0)}|Tu(\ovx_0)-u(\vx)|^p\dd\vx+c\rho_0^n|Tu(\ovx_0)|^p<\infty\\
    \Longrightarrow&
    u\in L^p(\dom_{\rho_0}(\ovx_0)).
\end{align*}
    Recall that $u\in\class{N}^{s(\cdot),p}(\dom)$ implies $u\in L^p_{\loc}(\dom)$. Since $Tu$ need not be bounded, however, we cannot conclude that $u\in L^p(\dom)$.
\end{enumerate}
\end{remark}
\begin{proof}
Put $\theta_0:=\theta_\bnd(\ovx_0)$. We may write
\begin{equation}\label{E:MainIntegralDecomp}
    \fslint{\dom_\rho(\ovx_0)}|u(\vx)-Tu(\ovx_0)|^p\dd\vx
    \le
    c\fslint{\dom_{\rho}(\ovx_0)}|u(\vx)-g(\vx)|^p\dd\vx+
    \fslint{\dom_\rho(\ovx_0)}|Tu(\vx_0)-g(\vx)|^p\dd\vx.
\end{equation}

For the first integral, the same argument used for the proof of \cref{C:LebPropTrace} can be used to show
\[
    \lim_{\rho\to0^+}\fslint{\dom_{\rho}(\ovx_0)}|u(\vx)-g(\vx)|^p\dd\vx
    =0.
\]
For the last integral in~\eqref{E:MainIntegralDecomp}, we need to introduce several additional definitions. Since $\omega<t-n(\theta_0-1)$, we may select $n(\theta_0-1)/p<\beta<(t-\omega)/p$. By \cref{T:uInWbp}, we find $Tu\in W^{\beta,p}(\bnd\cap\bll_{\rho'}(\ovx_0))$. \cref{C:LebPoints} implies
\begin{multline}\label{E:LebPntTu}
    \lim_{\rho\to0^+}\rho^{-n(\theta_0-1)}
        \fslint{\bnd\cap\bll_{\rho}(\ovx_0)}|Tu(\ovx)-Tu(\ovx_0)|^p\dd\haus^t(\ovx)\\
    \le\lim_{\rho\to0^+}\rho^{-p\beta}
        \fslint{\bnd\cap\bll_{\rho}(\ovx_0)}|Tu(\ovx)-Tu(\ovx_0)|^p\dd\haus^t(\ovx)=0.
\end{multline}
 With $\Lambda:=30C_\bnd\eta_1^{-1}$, and let $0<\rho_0\le\frac{1}{4}\Lambda^{-1}\rho'$ be given. For each $j\in\whls$, define $\rho_j:=2^{-j}\rho_0$ and
\begin{equation}\label{E:BiBoundedOverlap}
    E_j:=\lc\vx\in\dom_\rho(\ovx_0):\rho_{j+1}\le d_{\partial\dom}(\vx)<\rho_j\rc.
\end{equation}
We observe that $\dom_{\rho_0}(\ovx_0)=\bigcup_{j=0}^\infty E_j$. Define $r:\dom_{\rho_0}(\ovx_0)\to(0,\rho_2)$ by $r:=\frac{1}{4}d_{\partial\dom_{\rho_0}(\ovx_0)}(\vx)>0$. Besicovitch's covering theorem provides countable sets $I$ and $\{\vx_i\}_{i\in I_j}\subseteq D$ such that
\[
    \dom_\rho(\ovx_0)=\bigcup_{i\in I}B_i
    \nd
    \sup_{\vx\in\re}\sum_{i\in I}\chi_{B_i}(\vx)\le c(n).
\]
Here, we have put $B_i:=\bll_{r_i}(\ovx_i)$ and $r_i:=r(\vx_i)$, for $i\in I$. For each $j\in\whls$, define
\[
    I_j:=\lc i\in I:\vx_i\in E_j\rc.
\]
For each $j\in\nats$ and $i\in I_j$, we have $d_{\partial\dom}(\vx_i)<\rho_j$ and $\|\vx_i-\ovx_0\|_{\ren}<\rho_0$, so we may select $\ovx_i\in{\partial\dom}\cap\bll_{\rho_0}(\ovx_0)$ such that $\|\ovx_i-\vx_i\|_{\ren}<2d_{\partial\dom}(\vx_i)<\rho_{j-1}$. Using \cref{L:BoundedOverlap}, we may select a countable, in fact finite, set $K_j\subset\nats$ and $\lc\ovy_{j,k}\rc_{k\in K_j}\subseteq{\partial\dom}\cap\bll_{\rho_0}(\ovx_0)$ such that
\begin{equation}\label{E:B2RjBoundedOverlap}
    \partial\dom\cap\bll_{\rho_0}(\ovx_0)\subseteq\bigcup_{k\in K_j}\bnd'_{j,k}
        \subseteq\bnd\cap\bll_{2\rho_0}(\ovx_0)
    \nd
    \sup_{x\in\ren}\sum_{k\in K_j}\chi_{\bll_{2\Lambda\rho_j}(\ovy_{j,k})}(\vx)\le c(n).
\end{equation}
Here $\bnd'_{j,k}:=\partial\dom\cap\bll_{\rho_j}(\ovy_{j,k})$. Note that $\bigcup_{k\in K_j}\bnd'_{j,k}\subseteq\bnd\cap\bll_{2\rho_0}(\ovx_0)\subseteq\bnd\cap\bll_{\rho'}(\ovx_0)$.

Given $j\in\nats$ and $i\in I_j$, there exists a unique smallest $k\in K_j$ such that $\ovx_i\in\bnd_{j,k}$. Thus, defining
\[
    I_{j,k}:=\lc i\in I_j:\ovx_i\in\bnd'_{j,k}\bs\bigcup_{k'<k;k'\in K_j}\bnd'_{j,k'}\rc,
\]
we find $I_j:=\bigcup_{k\in K_j}I_{j,k}$. Set $U_{j,k}:=\bigcup_{i\in I_{j,k}}B_i$. Then $\{U_{j,k}\}_{k\in K_j}$ has bounded overlap. Indeed, since for each $i\in I_j$, there is only one $k\in K_j$ such that $i\in I_{j,k}$, we deduce that
\[
    \sup_{\vx\in\ren}\sum_{k\in K_j}\chi_{U_{j,k}}(\vx)
    \le
    \sup_{\vx\in\ren}\sum_{k\in K_j}\sum_{i\in I_{j,k}}\chi_{B_i}(\vx)
    =\sup_{\vx\in\ren}\sum_{i\in I_j}\chi_{B_i}(\vx)\le c(n).
\]
With $k\in K_j$ and $i\in I_{j,k}$, let $\vx\in U_{j,k}$ and $\ovy\in\bnd_{j,k}$ be given. Then, there exists an $i\in I_{j,k}$ such that $\vx\in B_i$. We see that
\[
    d_{\partial\dom}(\vx_i)<d_{\partial\dom}(\vx)+r_i
    <
    d_{\partial\dom}(\vx)+\smfrac{1}{4}d_{\partial\dom}(\vx_i)
    \Longrightarrow
    d_{\partial\dom}(\vx_i)<\smfrac{4}{3}d_{\partial\dom}(\vx)
\]
and
\begin{align*}
    \|\ovx-\vx\|
    \le&
    \|\ovx-\ovx_i\|_{\ren}
        +\|\ovx_i-\vx_i\|_{\ren}+\|\vx_i-\vx\|_{\ren}
    <
    2\rho_j+\rho_{j-1}+r_i\\
    <&
    5\rho_j\le 10d_{\partial\dom}(\vx_i)<15d_{\partial\dom}(\vx).
\end{align*}
Put $\lambda':=\frac{1}{15}\lambda_0$. We conclude that
\begin{equation}\label{E:BriSubsetQ}
    B_i\subseteq U_{j,k}
    \subseteq Q^1_{\frac{1}{15},5\rho_j}(\ovx)
    \subseteq Q_{\lambda',5\rho_j}(\ovx),
    \frl\ovx\in\bnd'_{j,k}.
\end{equation}

Now, we are prepared to bound the second integral in~\eqref{E:MainIntegralDecomp}. Since $\{U_{k,j}\}_{j\in\nats,k\in K_j}$ is a cover for $\dom_\rho$, we have
\begin{align}
\nonumber
    &\fslint{\dom_\rho(\ovx_0)}|g(\vx)-Tu(\ovx_0)|^p\dd\vx
    \le
    \frac{1}{|\dom_\rho(\ovx_0)|}\sum_{j=1}^\infty\sum_{k\in K_j}
        \slint{U_{j,k}}|g(\vx)-Tu(\ovx_0)|^p\dd\vx\\
\nonumber
    &\quad\le
    \frac{1}{|\dom_\rho(\ovx_0)|}\sum_{j=1}^\infty\sum_{k\in K_j}
        \frac{1}{|\bnd'_{j,k}|}\slint{\bnd_{j,k}}\slint{U_{j,k}}|g(\vx)-Tu(\ovx_0)|^p\dd\vx\dd\haus^t(\ovx)\\
\nonumber
    &\quad\le
    \frac{1}{|\dom_\rho(\ovx_0)|}\sum_{j=1}^\infty\sum_{k\in K_j}
    \frac{|U_{j,k}|}{|\bnd'_{j,k}|}\slint{\bnd_{j,k}}\fslint{U_{j,k}}|g(\vx)-Tu(\ovy)|^p\dd\vx\dd\haus^t(\ovx)\\
\label{E:DIntegralBound1}
    &\qquad\qqquad+
    \frac{1}{|\dom_\rho(\ovx_0)|}\sum_{j=1}^\infty\sum_{k\in K_j}
    |U_{j,k}|\fslint{\bnd_{j,k}}|Tu(\ovx)-Tu(\ovx_0)|^p\dd\vx\dd\haus^t(\ovx).
\end{align}
Select $\ell_j\in\whls$ such that $\eta_1^{\ell_j+1}\rho_\bnd\le 5\rho_j<\eta_1^{\ell_j}\rho_\bnd$. For the first integral in~\eqref{E:DIntegralBound1}, since $U_{j,k}\subseteq Q_{\lambda',5\rho_j}(\ovx)$, for each $\ovx\in\bnd'_{j,k}$, \cref{L:TugDiffBndG} implies
\[
    \slint{\bnd'_{j,k}}\fslint{U_{j,k}}|g(\vx)-Tu(\ovx)|^p\dd\vx
        \dd\haus^t(\ovx)\\
    \le
    c\vep_{\lambda''}^{-n-p}\lp\sum_{\ell=\ell_j}^\infty\lp\eta_1^{-\omega}
        \slint{\bnd'_{j,k}}G_{\eta_1^\ell\delta_\bnd}(\ovx)\dd\haus^t(\ovx)\rp^\frac{1}{p}\rp^p,
\]
with $\lambda'':=\eta_1^2\lambda'$. Continuing with \cref{L:supGBnd}, we obtain
\begin{align*}
    \slint{\bnd'_{j,k}}\fslint{U_{j,k}}|g(\vx)-Tu(\ovx)|^p\dd\vx
        \dd\haus^t(\ovx)
    \le&
    c\vep_{\lambda''}^{-n-p}\lp\sum_{\ell=\ell_j}^\infty\lp\eta_1^{\ell(t-\omega)}
        \nu^{s(\cdot),p}(\bll_{2\eta_1^\ell\delta_\bnd}(\bnd'_{j,k}))\rp^\frac{1}{p}\rp^p\\
    \le&
    C_3'\eta_1^{(\ell_j+1)(t-\omega)}
        \nu^{s(\cdot),p}(\bll_{2\eta_1^{\ell_j}\delta_\bnd}(\bnd'_{j,k}))\\
    \le&
    C_3'2^{-j(t-\omega)}\rho_0^{t-\omega}
        \nu^{s(\cdot),p}(\bll_{\Lambda\rho_j}(\bnd'_{j,k})),
\end{align*}
Here $C_3':=c\vep_{\lambda''}^{-n-p}\eta_1^{\omega-t}\lp1-\eta_1^{t-\omega}\rp^{-1}$ and, recalling the definition of $\Lambda$ above,
\[
    \Lambda\rho_j=30C_\bnd\eta_1^{-1}\rho_j
    \ge
    6C_\bnd\eta_1^{\ell_j}\rho_\bnd=2\eta_1^{\ell_j}\delta_\bnd
\]
The lower Ahlfor's regularity implies $\haus^t(\bnd'_{j,k})\ge A_\bnd^{-1}\rho_j^t=c2^{-jt}\rho_0^t$. Recalling~\eqref{E:BriSubsetQ} and the bounded overlap property in~\eqref{E:BiBoundedOverlap}, we find $|U_{j,k}|\le c\rho_j^n=c2^{-jn}\rho_0^n$.
Given $\vx\in\bll_{\Lambda\rho_j}(\bnd'_{j,k})$, we must have $\|\vx-\ovy_{j,k}\|_{\ren}\le(\Lambda+1)\rho_j\le2\Lambda\rho_j$, so $\bll_{\Lambda\rho_j}(\bnd'_{j,k})\subseteq\bll_{2\Lambda\rho_j}(\ovy_{j,k})\subseteq\bll_{4\Lambda\rho_0}(\ovx_0)$. Using the bounded overlap property for $\{\bll_{2\Lambda\rho_j}(\ovy_{j,k})\}_{k\in K_j}$, we deduce that
\begin{align*}
    &\sum_{k\in K_j}\frac{|U_{j,k}|}{\haus^t(\bnd'_{j,k})}\slint{\bnd'_{j,k}}
        \fslint{B_i}|g(\vx)-Tu(\ovx)|^p\dd\vx\dd\haus^t(\ovx)\\
    &\qqqquad\le
    cC_3'\lp2^{-j(t-\omega)}\rho_0^{t-\omega}\rp\lp2^{jt}\rho^{-t}\rp
        \sum_{k\in K_j}\lp|U_{j,k}|\nu^{s(\cdot),p}(\bll_{\Lambda\rho_j}(\bnd'_{j,k}))
            \rp\\
    &\qqqquad\le
    cC_3\lp2^{j\omega}\rho_0^{-\omega}\rp\lp2^{-jn}\rho^n\rp
        \sum_{k\in K_j}\nu^{s(\cdot),p}(\bll_{2\Lambda\rho_j}(\ovy_{j,k}))\\
    &\qqqquad\le
    cC_3'2^{-j(n-\omega)}\rho_0^{n-\omega}
        \nu^{s(\cdot),p}\lp\bigcup_{k\in K_j}\bll_{2\Lambda\rho_j}(\ovy_{j,k})\rp\\
    &\qqqquad\le
    cC_3'2^{-j(n-\omega)}\rho_0^{n-\omega}\nu^{s(\cdot),p}(\bll_{4\Lambda\rho_0}(\ovx_0))
\end{align*}
Also
\begin{align*}
    &\sum_{k\in K_j}|U_{j,k}|
        \fslint{\bnd'_{j,k}}|Tu(\ovx)-Tu(\ovx_0)|^p\dd\haus^t(\ovx)\\
    &\qqqquad\le
    c2^{jt}\rho_0^{-t}\sum_{k\in K_j}|U_{j,k}|\slint{\bnd'_{j,k}}|Tu(\ovx)-Tu(\ovx_0)|^p\dd\haus^t(\ovx)\\
    &\qqqquad\le
    c2^{-j(n-t)}\rho_0^{n-t}\slint{\bnd\cap\bll_{2\rho_0}(\ovx_0)}|Tu(\ovx)-Tu(\ovx_0)|^p\dd\haus^t(\ovx).
\end{align*}
Returning to~\eqref{E:DIntegralBound1}, we have established
\begin{align*}
    &\fslint{\dom_{\rho_0}(\ovx_0)}|g(\vx)-Tu(\ovx_0)|^p\dd\vx\\
    &\qquad\le
    \frac{c}{|\dom_{\rho_0}(\ovx_0)|}\lp C_3'\rho_0^{n-\omega}\nu^{s(\cdot),p}(\bll_{4\Lambda\rho_0}(\ovx_0))
        \sum_{j=1}^\infty2^{-j(n-\omega)}\rt\\
    &\qqqquad\lt+
        \rho_0^{n-t}\slint{\bnd\cap\bll_{2\rho}(\ovx_0)}|Tu(\ovx)-Tu(\ovx_0)|^p\dd\haus^t(\ovx)
        \sum_{j=1}^\infty 2^{-j(n-t)}\rp\\
    &\quad\le
    cC_3'\rho_0^{t-\omega-n(\theta_0-1)}
        \underbrace{\lp\rho_0^{-t}\nu^{s(\cdot),p}(\bll_{4\Lambda\rho_0}(\ovx_0))\rp}
            _{\le M_t(\ovx_0)<\infty}\\
    &\qqqquad+
        c\rho_0^{-n(\theta_0-1)}\fslint{\bnd\cap\bll_{2\rho_0}(\ovx_0)}|Tu(\ovx)-Tu(\ovx_0)|^p\dd\haus^t(\ovx).
\end{align*}
Recalling that $t-\omega>n(\theta_0-1)$ and~\eqref{E:LebPntTu}, we conclude that the upper bound above approaches vanishes as $\rho_0\to0^+$. Hence
\[
\lim_{\rho_0\to0^+}\fslint{\dom_{\rho_0}(\ovx_0)}|g(\vx)-Tu(\ovx_0)|^p\dd\vx=0,
\]
and the result is proved.
\end{proof}

\cref{T:TraceProp} follows from this section's main results.
\begin{proof}[Proof for \cref{T:TraceProp}]
Assumption (H3$''$) implies $\bnd=A_{\omega,\delta_\bnd}$, with $\omega=-\ul{\alpha}_\bnd<t$. Thus \cref{T:TraceProp}(a) follows from \cref{T:uInWbp}.

For \cref{T:TraceProp}(b), we assume that $\bnd=\partial\dom$, so $\partial\dom\cap\bll_{\rho_\bnd}(\ovx_0)\subseteq A_{\omega,\delta_\bnd}$. The result follows from \cref{T:LebProp}.
\end{proof}

\section{Appendix}\label{S:Appendix}

Here, we provide some details for the ``prickly'' snowflake $\dom\subseteq\re^2$ region presented in \cref{E:Domains}(\ref{E:Prickly}). The Koch snowflake curve can be identified as the attractor for a finte iterated function system (IFS), of similarity transforms, which facilitates many of the curves properties~\cite{Fal:03a}. Though the structure is similar, we require an infinite iterated function system (IIFS) to generate the curve $\bnd$ in \cref{E:Domains}(\ref{E:Prickly}). This makes the analysis of $\bnd$ substantially more difficult. One of the key ideas is that the images of certain compositions of the IIFS can be grouped together to mimic those produced by the IFS generating the Koch curve. Thi, in particular, allows us to establish $\haus^t(\bnd)>0$ and the Ahlfors-regularity with an argument similar to a standard one used for IFSs.

To organize the system and its compositions, we introduce several index sets. First, define $\set{I}':=\{3,4\}$, $\set{I}''_0:=\emptyset$, and $\set{I}''_j:=\{1,2\}^j$, for each $j\in\nats$.
We write $i'$ or $(i')$ for $\vei\in\set{I}'\times\set{I}''_0$. Define
\[
    \set{I}_j:=\set{I}'\times\set{I}''_j,\quad
    \set{I}'':=\bigcup_{j=0}^\infty\set{I}''_j,\nd \set{I}:=\set{I}'\times\set{I}''=\bigcup_{j=0}^\infty\set{I}_j.
\]
Our IIFS is $\class{F}:=\{\ve{f}_{\vei}\}_{\vei\in\set{I}}$, where each map $\ve{f}_{\vei}\in\cnt^\infty(\re^2;\re^2)$ is a contracting similarity, described below. To index the compositions of functions in $\class{F}$, we will use
\[
    \set{I}^*_j:=\set{I}^j=\underbrace{\set{I}\times\cdots\times\set{I}}_{j\text{ times}},
    \quad
    \set{I}^*_\infty:=\set{I}\times\set{I}\cdots,
    \nd
    \set{I}^*:=\bigcup_{j=1}^\infty\set{I}^*_j,
\]
so $\set{I}^*$ consists of all finite sequences of elements of $\set{I}$, and $\set{I}^*_\infty$ is the set of all infinite sequences. The lengths of $\vei''\in\set{I}''_j$, $\vei=(i',\vei'')\in\set{I}_j$, and $\vei^*=(\vei_1,\dots,\vei_{j^*})\in\set{I}^*_{j^*}$ are $|\vei''|:=j$, $|\vei|:=1+|\vei''|:=1+j$, and $|\vei^*|:=j^*$, respectively. We also define $\|\vei^*\|:=\sum_{k=1}^{|\vei^*|}|\vei_k|\ge|\vei^*|$. Given $\vei''\in\set{I}''$ and $0\le k\le|\vei''|$, we use $\vei''\big|_k$ for the truncation  of $\vei''$ to length $k$, so $\vei''\big|_k:=(i''_1,\dots,i''_k)$ and $\vei''\big|_0$ is the null vector. Similarly $\vei^*\big|_k:=(\vei_1,\dots,\vei_k)$, for $1\le k\le|\vei^*|$. The truncation for $\vei^*\in\set{I}^*$ is denoted similarly. We will use $\cdot\dplus\cdot$ to denote concatenation of two vectors. Given $\vei_1=(i_1',\vei_1'')\in\set{I}$

, so for example, if $\vei''_1\in\set{I}''_{j_1}$ and $\vei''_2\in\set{I}''_{j_2}$, then
\[
    \vei''_1\dplus\vei''_2
    =
    (i''_{1,1},\dots,i''_{1,j_1})\dplus(i''_{2,1},\dots,i''_{2,j_2})
    =(i''_{1,1},\dots,i''_{1,j_1},i''_{2,1},\dots,i''_{2,j_2})\in\set{I}''_{j_1+j_2}.
\]
If $\vei_1\in\set{I}$ and $\vei''_2\in\set{I}''$, then $\vei_1\dplus\vei''_2:=(i_1',\vei''_1\dplus\vei''_2)\in\set{I}_{|\vei''_1|+|\vei''_2|}$. Finally, we introduce a partial ordering on $\set{I}^*$ by writing $\vei_1^*=(\vei_{1,1},\dots,\vei_{1,j_1^*})\le(\vei_{2,1},\dots,\vei_{2,j_2^*})=\vei^*_2$ if
\[
    |\vei_1^*|=j_1^*\le j_2^*=|\vei_2^*|,\quad
    |\vei_{1,j_1^*}|\le|\vei_{2,j_1^*}|,\quad
    \vei^*_2\big|_{j_1^*-1}=\vei_1^*\big|_{j_1^*-1},
    \text{ and }\;
    \vei_{2,j_1^*}\big|_{|\vei_{1,j_1^*}|}=\vei_{1,j_1^*},
\]
so the first $(j_1^*-1)$ components of $\vei_1^*$ and $\vei_2^*$ are identical and $\vei_{1,j_1^*}$, the last component of $\vei^*_1$, is identical to the first $|\vei_{1,j_1^*}|$ components of $\vei_{2,j_1^*}$.


We now describe the similarity maps in $\class{F}$. For each $|\vei|\le3$, the action of $\ve{f}_{\vei}$ on the polygonal region $U_0$ is illustrated in \cref{F:InitialSet,F:ImageSets}.
\begin{figure}[h]
    \centering
    \parbox[b]{2.6in}{\hspace{0pt}
    \scalebox{1.1}{\includegraphics[trim=155pt 520pt 300pt 125pt,clip]{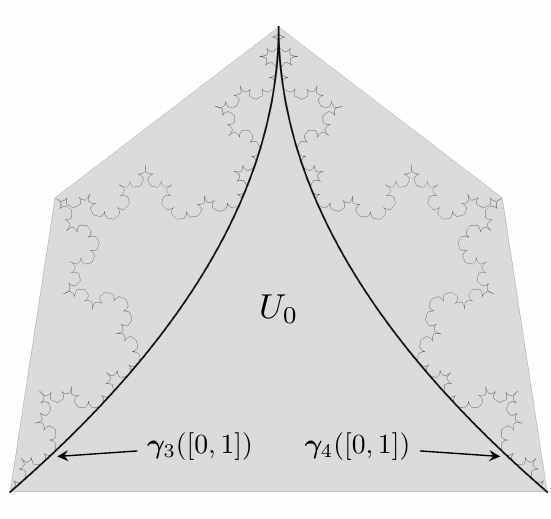}}
    \caption{Initial Set $U_0$}
    \label{F:InitialSet}}
    \hspace{20pt}
    \parbox[b]{2.6in}{
    \scalebox{1.1}{\includegraphics[trim=155pt 520pt 300pt 125pt,clip]{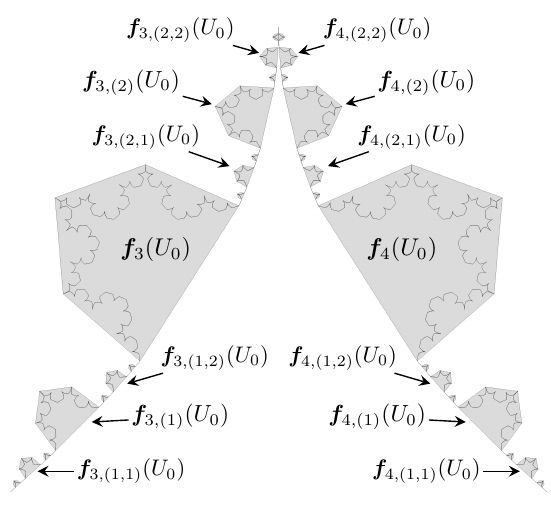}}
    \caption{Images of the Initial Set}
    \label{F:ImageSets}
    }
\end{figure}
To be more precise, let $\{I_{\vei''}\}_{\vei''\in I''}$ denote the open intervals that are removed during the construction of the middle-thirds Cantor set $\set{C}\subseteq[0,1]$. Given $\vei''\in\set{I}''$, the interval $I_{\vei''\dplus(1)}$ is the middle-third interval to the left of $I_{\vei''}$ and $I_{\vei''\dplus(2)}$ is the middle-third to the right, so for example,
\[
    I:=(\smfrac{1}{3},\smfrac{2}{3}),
    \:
    I_{(1)}=(\smfrac{1}{9},\smfrac{2}{9}),
    \:
    I_{(2)}=(\smfrac{7}{9},\smfrac{8}{9}),
    \:
    I_{(1,1)}=(\smfrac{1}{27},\smfrac{2}{27}),
    \:
    I_{(1,2)}=(\smfrac{7}{27},\smfrac{8}{27}),
    \:\text{etc..}
\]
Let $x_{\vei''}^-$ and $x_{\vei''}^+$ be the left and right endpoints of $I_{\vei''}$, respectively.
Suppose that the base of $U_0$ has unit length, is centered at $\ve{0}$, and is aligned with the first-coordinate axis. Further suppose that its height is at most $\sqrt{3}/2$. Let $\ve{\gamma}_{i'}\in\cnt([0,1];\re^2)$, with $i'\in\{3,4\}$, be parameterizations for the curves depicted in \cref{F:InitialSet} with the following property: for each $\vei''\in\set{I}''$, and $i''\in\{1,2\}$,
\[
    \|\ve{\gamma}_{i'}(x^+_{\vei''})-\ve{\gamma}_{i'}(x^-_{\vei''})\|_{\re^2}
    =
    3\|\ve{\gamma}_{i'}(x^+_{\vei''\dplus(i'')})
        -\ve{\gamma}_{i'}(x^-_{\vei''\dplus(i'')})\|_{\re^2},
    \frl i''\in\{1,2\}.
\]
We orient the curves so that $\ve{\gamma}_{i'}(1)$ corresponds to the point at the base of $U_0$ and $\ve{\gamma}_{i'}(0)$ is the cusp point. Put $L:=3\|\ve{\gamma}_{i'}(\frac{1}{3})-\ve{\gamma}_{i'}(\frac{2}{3})\|_{\re^2}<\frac{1}{2}\lp1+\sqrt{3}\rp$. Then
\[
    \|\ve{\gamma}_{i'}(x^+_{\vei''})-\ve{\gamma}_{i'}(x^-_{\vei''})\|_{\re^2}
    =\lp\smfrac{1}{3}\rp^{1+|\vei''|}L.
\]
The map $\ve{f}_{3,\vei''}$ is the unique orientation preserving similarity transform that maps the bottom right point $(\frac{1}{2},0)$ of $U_0$ to $\vx^-_{3,\vei''}:=\ve{\gamma}_3(x_{\vei''}^-)$ and maps the bottom left point $(-\frac{1}{2},0)$ to $\vx^+_{3,\vei''}:=\ve{\gamma}_3(x_{\vei''}^+)$. On the other hand, $\ve{f}_{4,\vei''}$ maps $(-\frac{1}{2},0)$ to $\vx^-_{4,\vei''}:=\ve{\gamma}_4(x_{\vei''}^-)$ and $(\frac{1}{2},0)$ to $\vx^+_{4,\vei''}:=\ve{\gamma}_4(x_{\vei''}^+)$. With $\vei=(i',\vei'')\in\set{I}$, the similarity ratio for $\ve{f}_{\vei}$ is $\sigma_{\vei}=\sigma_{i',\vei''}:=3^{-|\vei|}L<\frac{1}{6}(1+\sqrt{3})<\frac{1}{2}$.

The set $\bnd\subseteq\ov{U}_0$, considered in \cref{E:Domains}(\ref{E:Prickly}), is the unique compact set satisfying $\bnd=\ov{\bigcup_{\vei\in I}\ve{f}_{\vei}(\bnd)}$. Unlike finite iterated function systems, the closure is necessary~\cite{Fer:94a}, so $\bnd$ is not necessarily an invariant set for $\class{F}$. To establish the measure theoretic properties of $\bnd$ and take advantage of its self-similarity properties, we need to identify an invariant set $\bnd_0$ such that $\bnd\bs\bnd_0$ is a null set, with respect to an appropriate Hausdorff measure that will be determined later. Given $\vei^*\in\set{I}^*_j$ and $E\subseteq\ov{U}_0$, define
\[
    \ve{f}_{\vei_1^*}:=\ve{f}_{\vei_1}\circ\cdots\circ\ve{f}_{\vei_{j^*}},\quad
    E_{\vei^*}:=f_{\vei^*}(E),
    \nd
    \sigma_{\vei^*}:=\sigma_{\vei_1}\cdots\sigma_{\vei_{j^*}}=3^{-\|\vei^*\|}L^{-|\vei^*|}.
\]
If $E$ is closed and nonempty, then so is $E_{\vei^*}$. For convenience, we will use $U_{\vei^*}=U_{0,\vei^*}$. We  point out that $\class{F}$ satisfies the open set condition: for each $\vei_1,\vei_2\in\set{I}$, we find $U_{\vei_1},U_{\vei_2}\subseteq U_0$ and $U_{\vei_1}\cap U_{\vei_1}=\emptyset$ if $\vei_1\neq\vei_2$. More generally, if $\vei_1^*,\vei^*_2$, satisfies $\vei_2^*\big|_k\neq\vei_1^*\big|_k$, for some $1\le k\le\min\{|\vei_1^*|,|\vei_2^*|\}$, then $U_{\vei_2^*}\cap U_{\vei_1^*}=\emptyset$. Otherwise, either $U_{\vei_2^*}\subseteq U_{\vei_1^*}$ or $U_{\vei_1^*}\subseteq U_{\vei_2^*}$. We also observe that $E_{\vei^*}\subseteq E_{\vei^*\dplus\vei}$ and
\[
    \diam(E_{\vei^*\dplus\vei})
    =3^{-\|\vei^*\|-|\vei|}L^{1+|\vei^*|}\diam(E)
    <\lp\frac{2}{3}\rp^{|\vei^*|+|\vei|}\diam(E),
    \frl\vei\in\set{I}.
\]
Thus, for $\vei^*\in\set{I}^*_\infty$, the set $\ov{E}_{\vei^*}:=\bigcap_{j=1}^\infty\ov{E}_{\left.\vei^*\right|_j}$ is well-defined and a singleton. In~\cite{MauUrb:96a}, it is shown that $\bnd_0:=\bigcup_{\vei^*\in\set{I}^*_\infty}\ov{U}_{\vei^*}$ is an invariant set with respect to $\class{F}$; i.e.,
\[
    \bnd_0
    =\bigcup_{\vei\in\set{I}}\ve{f}_{\vei}(\bnd_0).
\]
We can now identify $\bnd\bs\bnd_0$. Set $\set{C}_{i'}:=\ve{\gamma}_{i'}(\set{C})$. Since $\set{C}_{i'}$ is closed and $\dist(\bnd_{i',\vei''},\set{C}_{i'})\to0$, as $|\vei''|\to\infty$, if $\{\vei''\}_{k=1}^\infty\subseteq\set{I}''$ consists of distinct indices, then $\lim_{k\to\infty}|\vei_k''|\to\infty$ and the limit points of any sequence $\{\vy_k\}_{k=1}^\infty$ satisfying $\vy_k\in\bnd_{i',\vei''_k}$ must belong to $\set{C}_{i'}$. Since each point in $\set{C}$ is a limit point of a sequence of endpoints of the intervals removed to generate $\set{C}$, we conclude that $\bnd\bs\bnd_0=\set{C}$. By Lemma~2.1 in~\cite{MauUrb:96a},
\[
    \bnd\bs\bnd_0=\bigcup_{i'\in\set{I}'}\lp\set{C}_{i'}\cup\bigcup_{\vei^*\in\set{I}^*}\ve{f}_{\vei^*}(\set{C}_{i'})\rp.
\]
Since $\haus^1(\set{C})=0$ and $\set{I}^*$ is countable, we conclude that $\haus^1(\bnd\bs\bnd_0)=0$. We also note that $\bnd_0$ is the set of so-called twist points for $\bnd$.


Next, we show that $\bnd_0$, and thus $\bnd$, has Hausdorff dimension $t>1$, and argue that its $\haus^t$-measure is finite. Since the system $\class{F}$ satisfies the open set condition, we may use a result from~\cite{Fer:94a} to conclude that the Hausdorff dimension for $\bnd_0$ is
\[
    t=\inf\lc\tau\in\re:\sum_{\vei\in\set{I}}
        \sigma^\tau_{\vei}\le 1\rc
    =\inf\lc\tau\in\re:L^\tau\sum_{j=1}^\infty
        \lp\frac{2}{3^\tau}\rp^j\le1\rc.
\]
In our case, the infimum is attained and is the unique solution to $2\lp L^t+1\rp=3^t$. The bounds for $L$ imply $t_1\le t\le t_2$, where $t_1:=(\ln4/\ln3)\approx1.26186$ and $t_2\approx1.49936$ is the unique solution to $2\lp(1+\sqrt{3})^{t_2}+2^{t_2}\rp=6^{t_2}$. Note, we could reduce the height of $U_0$, while leaving the base unchanged, so that $L=1$ and $t=(\ln4/\ln3)$. More generally, for any positive height, we find $L>\frac{1}{2}$ and $\bnd_0$ has Hausdorff dimension $t>1$. In any case, as $\haus^t(\bnd\bs\bnd_0)=0$, we conclude that, for all $E\in\class{B}(\ov{U}_0)\text{ and }\vei^*\in\set{I}^*$,
\[
    \haus^t(E\cap\bnd)=\haus^t(E\cap\bnd_0)
    \nd
    \haus^t(E\cap\bnd_{\vei^*})=\haus^t(E\cap\bnd_{0,\vei^*}).
\]
Thus, to establish Ahlfors-regularity for $\bnd$, it is sufficient prove it for $\bnd_0$.

To this end, we will take advantage of some results provided by~\cite{MauUrb:96a}. The class $\class{F}$ is clearly a conformal infinite iterated function system. The topological pressure, as defined in~\cite{MauUrb:96a}, is
\[
    \lim_{j\to\infty}\frac{1}{j}\ln\sum_{\vei^*\in\set{I}^*_j}\sigma^t_{\vei^*}
    =
    \lim_{j\to\infty}\frac{1}{j}
    \ln\lp\lp\sum_{\vei_1\in\set{I}}\sigma^t_{\vei_1}\rp
        \cdots\lp\sum_{\vei_j\in\set{I}}\sigma^t_{\vei_j}\rp\rp
    =
    \lim_{j\to\infty}\frac{\ln j}{j}=0.
\]
Therefore, there exists a $t$-conformal measure $m:\class{B}(\ov{U}_0)\to[0,1]$ on $\ov{U}_0$. That is, $m$ has the following properties: for each $E\in\class{B}(\ov{U}_0)$,
\begin{itemize}
    \item $m(E)=m(E\cap\bnd_0)$ and $m(\bnd_0)=1$;
    \item $m(E_{\vei})=m(\ve{f}_{\vei}(E))=\int_{E}\sigma_{\vei}^t\dd m=\sigma_{\vei}^tm(E)=L^t3^{-|\vei|t}m(E)$, for each $\vei\in\set{I}$;
    \item $m(\ov{U}_{\vei_1}\cap\ov{U}_{\vei_2})=0$ if and only if $\vei_1,\vei_2\in\set{I}$ satisfy $\vei_1\neq\vei_2$.
\end{itemize}
Lemma 4.2, in~\cite{MauUrb:96a}, implies the restriction of $\haus^t$ to $\bnd_0$, and thus $\bnd$, is absolutely continuous with respect to $m$ and finite. In fact, by inspecting the proof, we may deduce $\haus^t(\bnd)=\haus^t(\bnd_0)\le \diam(U_0)^t$. In any case, as we are only interested in the $\haus^t$-properties and $m$-properties of $\bnd$, throughout the remainder of this section, we need not distinguish between $\bnd_0$ and $\bnd$.

We now work to establish a positive lower bound for $\haus^t(\bnd)$. As stated earlier, a key idea is to group subcollections of $\set{I}^*$ in such a way that we may work with $\class{F}$ in a manner similar to an IFS (properties~(\ref{E:KochProp}) and~(\ref{E:UComp})) below). Our objective is to show that there exists a constant $0<c<\infty$ such that $m(E)\le c\diam(E)^t$, for all $E\in\class{B}(\ov{U}_0)$. Then, the Mass Distribution Principle~\cite{Fal:03a} delivers the lower bound $\haus^t(\bnd)\ge c^{-1}m(\bnd)=c^{-1}$. Iterating the properties for $m$ stated above yields
\begin{itemize}
    \item $m(E_{\vei^*})=m(\ve{f}_{\vei^*}(E)\cap\bnd_0)=\int_{E}\sigma_{\vei^*}^t\dd m=\sigma_{\vei^*}^tm(E)$, for all $E\in\class{B}(\ov{U}_0)$;
    \item $m(\ov{U}_{\vei^*_1}\cap\ov{U}_{\vei_2^*})=0$ if and only if $\vei_2^*\big|_k\neq\vei^*_1\big|_k$, for some $1\le k\le\min\{|\vei^*_1|,|\vei^*_2|\}$.
\end{itemize}
Put $r_0:=\sup\{r>0:\bll_r(\vx)\subseteq U_0\}>0$. Select $\vx_0\in U_0$ such that $\bll_{r_0}(\vx_0)\subseteq U_0$. For each $\vei^*\in\set{I}^*$,
\[
    \diam(U_{\vei^*})=\sigma_{\vei^*}D_0
    \nd
    \bll_{\sigma_{\vei^*}r_0}(\ve{f}_{\vei^*}(\vx_0))
    \subseteq
    U_{\vei^*}
    \subseteq
    \bll_{\sigma_{\vei^*}D_0}(\ve{f}_{\vei^*}(\vx_0)).
\]
Define $\wh{I}^*(\vei^*):=\lc\wh{\vei}^*\in\set{I}^*:\vei^*\le\wh{\vei}^*\rc$,
\[
    \wh{U}_{\vei^*}
    :=
    \bigcup_{\wh{\vei}^*\in\wh{\set{I}}^*(\vei^*)}U_{\wh{\vei}^*},
    \nd
    \wh{\bnd}_{\vei^*}
    :=
    \bigcup_{\wh{\vei}^*\in\wh{\set{I}}^*(\vei^*)}\bnd_{\wh{\vei}^*}.
\]
For convenience, given $\vei''\in\set{I}''$, we will use $\vei^*\dplus(i'')=(\vei_1,\dots,\vei_{|\vei^*|}\dplus(\vei''))\in\set{I}^*_{|\vei^*|}$. Recall that $\vei^*\dplus(i')=(\vei_1,\dots,\vei_{|\vei^*|},(i'))\in\set{I}^*_{1+|\vei^*|}$, for $i'\in\set{I}'$. In either case, $\|\vei^*\dplus(i)\|=1+\|\vei^*\|$. We then may write
\begin{enumerate}[(a)]
    \item\label{E:KochProp} $\wh{U}_{\vei^*}=\bigcup_{i=1}^4\wh{U}_{\vei^*\dplus i}$ and $\wh{\bnd}_{\vei^*}=\bigcup_{i=1}^4\wh{\bnd}_{\vei^*\dplus i}$;
    \item\label{E:URep} $\wh{U}_{\vei^*}=\bigcup_{\vei''\in\set{I}''}U_{\vei^*\dplus\vei''}$ and $\wh{\bnd}_{\vei^*}=\bigcup_{\vei''\in\set{I}''}\bnd_{\vei^*\dplus\vei''}$;
    \item\label{E:UComp} $\wh{U}_{\vei_2^*}\subseteq\wh{U}_{\vei_1^*}$ if and only if $\vei_1^*\le\vei_2^*$. Otherwise $\wh{U}_{\vei_1^*}\cap\wh{U}_{\vei_2^*}=\emptyset$;
    \item\label{E:UBnds} $\bll_{\sigma_{\vei^*}r_0}(\ve{f}_{\vei^*}(\vx_0))\subseteq\wh{U}_{\vei^*}\subseteq\bll_{3\sigma_{\vei^*}}D_0(\ve{f}_{\vei^*}(\vx_0))$.
\end{enumerate}
This last property follows from $U_{\vei^*}\subseteq\wh{U}_{\vei^*}$ and the observation that $\diam(\wh{D}_{\vei^*})$ cannot exceed the diameter of the set produced by arranging the sets $\{U_{\vei^*\dplus\vei''}\}_{\vei''\in\set{I}''}$ adjacent to each other so that their bases are adjoined along a common line, as would occur if $\bnd$ was the Koch snowflake curve. In this case
\[
    \diam(\wh{U}_{\vei^*})
    \le
    \sum_{\vei''\in\set{I}''}\diam(U_{\vei^*\dplus\vei''})
    =
    D_0\sum_{\vei''}\sigma_{\vei^*\dplus\vei''}
    =
    D_0\sigma_{\vei^*}\sum_{k=0}^\infty\lp\frac{2}{3}\rp^k=3\sigma_{\vei^*}D_0.
\]
Since each pair of sets in $\{\bnd_{\vei^*\dplus\vei''}\}_{\vei''\in\set{I}''}$ can only intersect in an $m$-null set, property~(\ref{E:URep}) implies
\begin{equation}\label{E:mMeasBnd}
    m(\wh{\bnd}_{\vei^*})
    =
    \sum_{\vei''\in\set{I}''}m(\bnd_{\vei^*\dplus\vei''})
    =
    m(\bnd)\sum_{\vei''\in\set{I}''}\sigma^t_{\vei^*\dplus\vei''}
    =
    \sigma^t_{\vei^*}\sum_{k=0}^\infty\lp\frac{2}{3^t}\rp^k
    =
    \lp\frac{3^t}{3^t-2}\rp\sigma^t_{\vei^*}.
\end{equation}
The scaling properties for $\haus^t$ similarly produce
\begin{equation}\label{E:HausBnd}
    \haus^t(\wh{\bnd}_{\vei^*})=\lp\frac{3^t}{3^t-2}\rp\sigma^t_{\vei^*}\haus^t(\bnd).
\end{equation}
Next, for $\vei^*=(\vei_1,\dots,\vei_{j^*})\in\set{I}^*$ satisfying $\|\vei^*\|>1$, put
\[
    \ov{\vei}^*:=\lc\begin{array}{ll}
        (\vei_1,\dots,\vei_{j^*-1})\dplus\vei_{j^*}\big|_{|\vei_{j^*}|-1}, &\text{ if }|\vei_{j^*}|>1,\\
        (\vei_1,\dots,\vei_{j^*-1}), &\text{ if }|\vei_{j^*}|=1,
    \end{array}\rt
\]
so $\|\ov{\vei}^*\|=\|\vei^*\|-1$. For each $0<\rho\le\frac{1}{3}$ and $\vx\in\ov{U}_0$, set
\[
    \set{I}^*_\rho
    :=
    \lc\vei^*\in\set{I}^*
        :\sigma_{\vei^*}D_0\le\rho<\sigma_{\ov{\vei}^*}D_0\rc
\]
and
\[
    \set{I}^*_\rho(\vx)
    :=
    \{\vei^*\in\set{I}_\rho^*:\wh{\bnd}_{\vei^*}\cap\bll_\rho(\vx)\neq\emptyset\}.
\]
For any $\vei^*\in\set{I}^*$, we have $\bnd_{\vei^*}\subseteq\wh{\bnd}_{\vei^*}\subseteq\wh{\bnd}_{\ov{\vei}^*}$ and $\sigma_{\ov{\vei}^*}\ge 3\sigma_{\vei^*}$. Thus, for each $\vx\in\bnd$ and $\vei_1^*\in\set{I}^*$, satisfying $\sigma_{\vei_1^*}D_0\le\rho$ and $\bnd_{\vei_1^*}\cap\bll_\rho(\vx)\neq\emptyset$, there exists a $\vei_2^*\in\set{I}^*_\rho(\vx)$, possibly equal to $\vei_1^*$, such that $\bnd_{\vei_1^*}\subseteq\wh{\bnd}_{\vei_2^*}$. This implies
\[
    \bnd\cap\bll_\rho(\vx)
    \subseteq
    \ov{\bigcup_{\vei^*\in\set{I}^*_\rho(\vx)}\wh{\bnd}_{\vei^*}}
    \subseteq
    (\bnd\bs\bnd_0)\cup\bigcup_{\vei^*\in\set{I}^*_\rho(\vx)}\wh{\bnd}_{\vei^*},
\]
and thus,
\begin{equation}\label{E:mBndNeighb}
    m(\bnd\cap\bll_\rho(\vx))
    \le
    \sum_{\vei^*\in\set{I}^*_\rho(\vx)}m(\wh{\bnd}_{\vei^*}).
\end{equation}
With $E\in\class{B}(\ov{U}_0)$, we now show that $m(E)\le c\diam(E)^t$, with $c$ independent of $E$. We may assume that $0<\diam(E)<\frac{1}{6}$, so there exists $\vx\in\bnd$ and $0<\rho\le2\diam(E)\le\frac{1}{3}$ such that $E\cap\bnd\subseteq\bnd\cap\bll_\rho(\vx)$. By property~(\ref{E:UBnds}), for each $\vei^*\in\set{I}^*_\rho(\vx)$, we must have
\[
    |\wh{U}_{\vei^*}|\ge(\sigma_{\vei^*}r_0)^n|\bll_1|
    \ge\lp\frac{1}{3L}\sigma_{\ov{\vei}^*}r_0\rp^n|\bll_1|
    \ge\lp\frac{r_0}{3LD_0}\rp^n|\bll_\rho|
\]
and $\wh{U}_{\vei^*}\subseteq\bll_{3\rho}(\vx)$. By property~(\ref{E:UComp}) and the definition of $\set{I}^*_\rho$, the family $\{\wh{U}_{\vei^*}\}_{\vei^*\in\set{I}^*_\rho(\vx)}$ must consist of mutually disjoint sets. Consequently, there is an upper bound on the number of elements in $\set{I}^*_\rho(\vx)$. In fact, we conclude that
\[
    \card(\set{I}^*_\rho(\vx))
    \le
    \sup_{\vei^*\in\set{I}^*_\rho(\vx)}\lp\frac{|\bll_{3\rho}|}{|\wh{U}_{\vei^*}|}\rp
    \le
    \lp\frac{9LD_0}{r_0}\rp^n.
\]
Using~\eqref{E:mMeasBnd} and~\eqref{E:mBndNeighb}, we obtain
\begin{align}
\nonumber
    m(\bll_\rho(\vx))
    =&
    m(\bll_\rho(\vx)\cap\bnd)
    \le
    \lp\frac{9LD_0}{r_0}\rp^n\max_{\vei^*\in\set{I}^*_\rho(\vx)}m(\wh{\bnd}_{\vei^*})\\
\label{E:UppRegMass}
    \le&
    m(\bnd)\lp\frac{9LD_0}{r_0}\rp^n\!\lp\frac{3^t}{3^t-2}\rp
        \max_{\vei^*\in\set{I}^*_\rho(\vx)}\!\sigma^t_{\vei^*}
    \le
    \lp\frac{9LD_0}{r_0}\rp^n\!\lp\frac{3^t}{3^t-2}\rp\lp\frac{\rho}{D_0}\rp^t.
\end{align}
Since $\rho^t\le2^t\diam(E)^t$, the Mass Distribution Principle~\cite{Fal:03a} implies
\[
    \haus^t(\bnd)\ge
    \lp\frac{3^t-2}{3^t}\rp\lp\frac{r_0}{9LD_0}\rp^n\lp\frac{D_0}{2}\rp^t.
\]
We have thus far shown that $0<\haus^t(\bnd)<\infty$.


Next, we work towards establishing the Ahlfors-regulartiy for $\bnd$. First, with $m$ replaced by $\haus^t$, we can incorporate~\eqref{E:HausBnd} into the exact same argument that established~\eqref{E:UppRegMass} to prove the upper Ahlfors-regularity of $\bnd$. Thus, we need only show lower Ahlfors-regularity. Let $\vx\in\bnd$ and $0<\rho\le\frac{1}{3}$ be given. For each $k\in\nats$, there exists $\vei^*_k\in\set{I}^*$ such that $\|\vei_k^*\|=k$ and $\vx\in\ov{\wh{\bnd}}_{\vei_k^*}\subseteq\ov{\wh{U}}_{\vei_k^*}$. Since $\lim_{k\to\infty}\sigma_{\vei^*_k}=0$, we may select $\vei^*\in\set{I}^*$ such that $\vx\in\ov{\wh{\bnd}}_{\vei^*}$ and $3\sigma_{\vei^*}D_0\le\frac{1}{2}\rho<3\sigma_{\ov{\vei}^*}D_0$. By property~(\ref{E:UBnds}), $\diam(\ov{\wh{\bnd}}_{\vei^*})\le\frac{1}{2}\rho$, and therefore, $\ov{\wh{\bnd}}_{\vei^*}\subseteq\bll_\rho(\vx)\cap\bnd$. By~\eqref{E:HausBnd},
\begin{align*}
    \haus^t(\bll_\rho(\vx)\cap\bnd)\ge\haus^t(\wh{\bnd}_{\vei^*})
    =&
    \lp\frac{3^t}{3^t-2}\rp\sigma^t_{\vei^*}\haus^t(\bnd)
    \ge
    \lp\frac{3^t}{3^t-2}\rp\lp\frac{\sigma_{\ov{\vei}^*}}{3L}\rp^t\haus^t(\bnd)\\
    \ge&
    \lp\frac{1}{6D_0L}\rp^t\lp\frac{1}{3^t-2}\rp\haus^t(\bnd)\rho^t.
\end{align*}
This concludes the proof.

Finally, we discuss the other claims made in \cref{E:Domains}(\ref{E:Prickly}). The domain $\dom$ is bounded by $\bnd$ and the line segment joining the points $(\pm\frac{1}{2},0)$. Let $T_0$ be the open subset of $\dom$ bounded by the curves $\ve{\gamma}_3([0,1])$ and $\ve{\gamma}_4([0,1])$. For some $\theta_0\ge 1$, this wedge-shaped region is congruent to $\dom^{\theta_0}$, as defined in \cref{E:Domains}(\ref{E:Proto}). It is clear that $\dom$ satisfies hypotheses (H1) and (H2) with $\theta_\bnd\equiv\theta_0$. Put $\ovy_0:=\ve{\gamma}_3(0)=\ve{\gamma}_4(0)\in\partial T_0$. There is, however, no approach region for the cusp point $\ovy_0$ that has both the $(\eta,\theta)$-corkscrew and the $(C,\theta)$-connectedness properties, for any $0<\eta<1\le\theta<\theta_0$ and $C\ge 1$. If there were, then we could use the symmetry of $\dom$ and an argument similar to the one used for \cref{L:Hypotheses}(\ref{L:ThinCorkscrew}), to conclude that the set $T_0$ must be a $(\eta,\theta)$-corkscrew region for $\ovy_0$, which is a contradiction. By the self-similarity of $\bnd$, this must also true for each point in $\ovy\in\Lambda:=\{\ovy_0\}\cup\bigcup_{\vei^*\in\set{I}^*}\ve{f}_{\vei^*}(\ovy_0)$. One easily sees that given any $\ovx\in\bnd$ and $\rho>0$, there exists $\ovy\in\bnd\cap\Lambda\cap\bll_\rho(\ovx)$. This implies $\dom_\rho(\ovx)$ fails to satisfy (H1) and (H2), for any $1\le\theta<\theta_0$, and thus, in particular, there is no $1$-sided NTA neighborhood of any $\ovx\in\bnd$ (see~\cref{R:HypConn}(\ref{R:H2Rem})).  We similarly conclude that there are no locally uniform neighborhoods of points in $\bnd$.

\bibliographystyle{spmpsci}      

\end{document}